\documentclass[11pt]{article}

\usepackage[T1]{fontenc}
\usepackage[utf8]{inputenc}
\usepackage{lmodern}
\usepackage{microtype}

\usepackage[margin=0.88in]{geometry}

\usepackage{amsmath,amssymb,amsthm,mathtools}

\usepackage{booktabs}
\usepackage{tabularx}
\usepackage{enumitem}
\usepackage{seqsplit}

\usepackage{xcolor}
\usepackage[
    colorlinks=true,
    linkcolor=blue!55!black,
    citecolor=blue!55!black,
    urlcolor=blue!55!black
]{hyperref}

\hypersetup{
    pdftitle={A Proof of Fraenkel's Conjecture},
    pdfauthor={Hu Tan and Ying Zhang},
    pdfsubject={Fraenkel's conjecture, Beatty partitions, and Fourier rigidity}
}

\title{A Proof of Fraenkel's Conjecture}

\author{}
\date{}

\makeatletter
\renewcommand{\maketitle}{%
  \begin{center}

    \vspace*{-1.0em}

    {\LARGE
      \@title
      \par
    }

    \vspace{1.15em}

    {\large
      Hu Tan\textsuperscript{1}
      \hspace{2.0em}
      Ying Zhang\textsuperscript{2}
      \par
    }

    \vspace{0.85em}

    \begin{minipage}[t]{0.45\textwidth}
      \centering
      \small
      \textsuperscript{1}Academy of Mathematics and Systems Science\\
      Chinese Academy of Sciences\\
      Beijing 100190, China\\[0.25em]
      \footnotesize
      \href{mailto:tanhu2020@amss.ac.cn}
      {\texttt{tanhu2020@amss.ac.cn}}
    \end{minipage}
    \hspace{0.035\textwidth}
    \begin{minipage}[t]{0.45\textwidth}
      \centering
      \small
      \textsuperscript{2}School of Mathematical Sciences\\
      Soochow University\\
      Suzhou 215006, China\\[0.25em]
      \footnotesize
      \href{mailto:yzhang@suda.edu.cn}
      {\texttt{yzhang@suda.edu.cn}}
    \end{minipage}

    \vspace{0.9em}

  \end{center}

  \vspace{0.7em}
}
\makeatother

\newtheorem{theorem}{Theorem}[section]
\newtheorem{proposition}[theorem]{Proposition}
\newtheorem{lemma}[theorem]{Lemma}
\newtheorem{corollary}[theorem]{Corollary}

\theoremstyle{definition}

\theoremstyle{remark}

\newcommand{\Z}{\mathbb Z}
\newcommand{\U}{(\mathbb Z/Q\mathbb Z)^\times}
\newcommand{\B}{\mathcal B}
\newcommand{\normQ}[1]{\left\lVert #1\right\rVert_Q}
\newcommand{\ii}{\mathrm i}
\newcommand{\Harm}{\mathbf H}

\begin{document}
\maketitle

\begin{abstract}
Fraenkel's conjecture asserts that a partition of the integers into at least three Beatty sequences with distinct moduli has the binary densities \(1,2,4,\ldots,2^{m-1}\), normalized by \(2^m-1\).

We prove the conjecture through a dimension-free intermediate statement:
every such partition contains a component of density at least \(1/3\).

After reducing the partition to primitive common-period data, Fourier
cancellation produces a finite inverse-sine system.  We prove that no such
system can exist when every density is below \(1/3\).  The proof combines a
divisor-concentration identity with uniform analytic estimates and three
exact finite verifications, all carried out with integer or rational
arithmetic. 

The component supplied by the density bound has mean spacing at most three. Deleting it preserves balance, and every surviving periodic balanced set is
again a rational Beatty set. Induction determines the surviving binary
scales, while a two-sequence disjointness criterion forces the deleted density
to be the next binary scale. This yields the asserted density pattern.

An eventual Beatty representation of balanced binary indicators extends the
classification to one-sided balanced sequences on at least three letters
with distinct positive densities, proving the balanced-sequence conjecture
of Altman, Gaujal, and Hordijk.
\end{abstract}


\section{Introduction and main result}

An exact partition by Beatty sequences is governed by continuous parameters,
but the conjectured answer is entirely discrete.  Fraenkel's conjecture says
that an exact partition into \(m\ge3\) Beatty sequences with pairwise distinct
moduli can occur only in one binary pattern.  The difficulty is therefore one
of rigidity: the covering identity must recover the whole pattern, not merely
produce examples.

For \(\alpha>1\) and \(\beta\in\mathbb R\), write
\[
 S(\alpha,\beta)=\{\lfloor n\alpha+\beta\rfloor:n\in\Z\}.
\]
This set has asymptotic density \(1/\alpha\) and mean spacing \(\alpha\).
Thus Fraenkel's conjecture concerns exact partitions of \(\Z\) into \(m\)
such sets with pairwise distinct moduli.  We use the bi-infinite formulation
because it is stable under translation and, after rational normalization,
is naturally periodic.

Fraenkel exhibited the binary family in 1973 and conjectured the weaker
divisibility statement that, in any such partition with at least three
components, two moduli have an integral ratio~\cite{Fraenkel1973}.
The stronger binary uniqueness formulation, now usually known as
Fraenkel's conjecture, was explicitly recorded by Erd\H{o}s and Graham
\cite[p.~19]{ErdosGraham1980}; see also the historical clarification in
\cite[Conjecture~1.4 and footnote~1]{DamasdiEtAl2026}.
The conjecture was proved for \(m\le6\) by Tijdeman and for \(m=7\)
by Bar\'at and Varj\'u~\cite{Tijdeman2000,BaratVarju2003}.
Dam\'asdi, Frankl, Pach, and P\'alv\"olgyi still record the general case
\(m\ge8\) as open~\cite[Conjecture~1.4 and the following paragraph]{DamasdiEtAl2026}.
We prove the conjecture in full.

The established approaches left a precise interface unresolved.  Tijdeman's
deletion method would complete the induction once a component of mean spacing
at most three were found~\cite[Remark~5, p.~479]{TijdemanBalanced2000}.
Under the contrary hypothesis, the Graham--O'Bryant transform produces
inverse-sine row inequalities on a minimum-gcd Fourier layer
\cite[Theorem~1.1 and Section~4.2]{GrahamOBryant2005}.  What was missing was
a dimension-free mechanism forcing one of those rows below one.  The
one-third theorem stated below supplies exactly this bridge.

\begin{theorem}[Fraenkel's conjecture]\label{thm:main}
Let \(m\ge3\).  Suppose that the sets
\(S(\alpha_i,\beta_i)\), \(1\le i\le m\), partition \(\Z\), and that
the moduli \(\alpha_i>1\) are pairwise distinct.  Then
\[
 \{\alpha_1,\ldots,\alpha_m\}
 =\left\{\frac{2^m-1}{2^j}:0\le j<m\right\}.
\]
\end{theorem}

For \(m=3\), the theorem predicts the primitive density pattern
\(\{1/7,2/7,4/7\}\), or equivalently the moduli
\(\{7,7/2,7/4\}\).  It is realized by
\(S(7,0)\), \(S(7/2,2)\), and \(S(7/4,3)\): modulo \(7\), these sets occupy
\(\{0\}\), \(\{2,5\}\), and \(\{1,3,4,6\}\), respectively, and hence
partition \(\Z\).  The higher-dimensional pattern continues the same
doubling rule.

The point at which the analytic and combinatorial arguments meet is the
following density bound, independent of the number of parts.

\begin{theorem}[One-third theorem]\label{thm:one-third-main}
Under the hypotheses of Theorem~\ref{thm:main},
\[
 \max_i\frac1{\alpha_i}\ge\frac13.
\]
\end{theorem}

Equivalently, some component has mean spacing at most three.  Tijdeman
identified precisely this conclusion as the missing input for a general
deletion induction.

The classification also settles the balanced-sequence conjecture of
Altman, Gaujal, and Hordijk~\cite[Conjecture~2.25]{AltmanGaujalHordijk2000},
restated as Conjecture~2.2 in~\cite{DamasdiEtAl2026}.
A one-sided infinite word is \emph{balanced} if, for each letter, its
numbers of occurrences in any two finite factors of the same length differ
by at most one.  Its letter densities are the limiting proportions in
initial factors.

\begin{corollary}[Balanced-sequence classification]\label{cor:balanced}
Let \(m\ge3\) and \(\rho_1>\cdots>\rho_m>0\).  A balanced one-sided infinite
word on \(m\) letters with respective densities \(\rho_1,\ldots,\rho_m\) exists
if and only if
\(\rho_i=2^{m-i}/(2^m-1)\) for every \(1\le i\le m\).
\end{corollary}

Section~\ref{subsec:balanced-corollary} proves this consequence by showing
that each balanced binary indicator agrees with a Beatty indicator on a
tail.  Graham's rationality theorem then gives a common periodic extension
to which Theorem~\ref{thm:main} applies.

\subsection{Proof architecture}

The proof uses four established inputs: Graham's rationality reduction
\cite{Graham1973}; the rational Beatty transform of Graham and O'Bryant
\cite[Theorem~1.1]{GrahamOBryant2005}; Fraenkel's rescaling
\cite[Lemma~6]{Fraenkel1973} and Tijdeman's balanced deletion
\cite{TijdemanBalanced2000}; and the Morikawa--Simpson criterion for two
disjoint rational Beatty sets
\cite{Morikawa1985,Simpson2004}.  The new work is the dimension-free
subcritical rigidity theorem, its exact finite closure, and the rigid
extension that turns the one-third bound into the full classification.

The central analytic input is the subcritical inverse-sine rigidity theorem,
Theorem~\ref{thm:SC}.  Informally, after the minimum-gcd normalization the
relevant numerators are distinct units modulo a denominator \(Q\), each is
smaller than \(Q/3\), and their total is at most \(Q\).  The theorem says that
one of their inverse-sine interaction rows must have total mass below one,
contradicting the lower bound imposed by Fourier cancellation.

The proof of this statement has a single organizing mechanism.  Let \(s\) be
the smallest relevant numerator.  An exact slack identity shows that a row
with mass at least one can place very little numerator weight away from
divisors of \(Q-s\) and \(Q+s\).  For \(s>1\), the least such divisor creates
a fixed gap in another row.  For \(s=1\), that gap may be hidden by a dyadic
predecessor chain, and a harmonic propagation argument takes its place.
These uniform estimates reduce the proof to three finite arithmetic systems.
Each finite inequality is a one-sided upper bound that preserves the
implication from a genuine counterexample.  Exact arithmetic then excludes
the bounded cases; Appendix~\ref{app:verification} states the corresponding
finite verification statements.

Once Theorem~\ref{thm:one-third-main} is available, the remainder is
combinatorial.  Deleting a component of mean spacing at most three preserves
balance.  Periodicity then identifies each compressed indicator as a rational
Beatty set, so induction makes the surviving densities binary up to a common
scale.  The Morikawa--Simpson two-sequence criterion forces the deleted
component to be the next binary scale.

Sections~\ref{sec:normalization}--\ref{sec:fourier} set up these two
interfaces and reduce the analytic part to Theorem~\ref{thm:SC}.
Sections~\ref{sec:slack}--\ref{sec:s-one} prove that theorem, first through
the common slack mechanism and then in the cases \(s>1\) and \(s=1\).
Section~\ref{sec:induction} returns to Beatty words and completes the
classification.

Readers interested primarily in the reconstruction may read
Sections~\ref{sec:normalization} and~\ref{sec:deletion}, take
Theorem~\ref{thm:one-third-main} as the analytic input, and continue with
Section~\ref{sec:induction}.  The intervening sections prove precisely that
input.

\section{Rational and primitive normalization}\label{sec:normalization}

The Fourier argument requires a common finite period, whereas the conjecture
is stated with arbitrary real shifts and a priori unrelated real moduli.  We
remove these two degrees of freedom explicitly.  The resulting integer
shifts, normalized densities, and gcd layers will also be used in the final
induction.

\subsection{From real data to a common period}

Graham's corollary states that an eventual covering family with at least
three members and an irrational modulus must contain two equal
moduli~\cite[Corollary, p.~358]{Graham1973}.  Since our moduli are pairwise
distinct, they are therefore all rational.  To see that the corollary
applies, fix \(i\).  The sequence
\(\lfloor n\alpha_i+\beta_i\rfloor\) is strictly increasing because
\(\alpha_i>1\).  Hence there is an integer \(N_i\) such that its positive
terms are exactly those with \(n\ge N_i\).  Reindexing
\(n=N_i,N_i+1,\ldots\) turns the positive part into a standard one-sided
Beatty sequence.  Since the original sets partition \(\Z\), these positive
parts partition the positive integers.

For rational data, every real shift determines an integer-translated rational
Beatty set.  We record the short calculation because it also identifies the
translation.  Let the reduced modulus be \(q/p\), with \((p,q)=1\), and
write
\[
 \beta=b+\gamma,\qquad b\in\Z,\qquad 0\le\gamma<1.
\]
If \(p=1\), then for every \(n\in\Z\),
\[
 \left\lfloor\frac{nq}{p}+b+\gamma\right\rfloor
 =\lfloor nq+b+\gamma\rfloor=nq+b.
\]
Thus \(S(q,\beta)=\{nq+b:n\in\Z\}\) already has the required
integer-translate form.  We henceforth assume \(p>1\).
There is a unique integer \(h\in\{0,1,\ldots,p-1\}\) such that
\[
 \frac hp\le\gamma<\frac{h+1}{p}.
\]
Because \((p,q)=1\), multiplication by \(q\) permutes the residue classes
modulo \(p\).  Hence we may choose \(j\in\Z\) with
\(jq\equiv h\pmod p\), and then write
\[
 jq=h+tp\qquad(t\in\Z).
\]
After replacing the index \(n\) by \(n-j\), we obtain
\begin{align*}
 \frac{(n-j)q}{p}+b+\gamma
 &=\frac{nq}{p}+(b-t)+\left(\gamma-\frac hp\right),\\
 0&\le\gamma-\frac hp<\frac1p.
\end{align*}
The fractional part of \(nq/p\) is one of
\(0,1/p,\ldots,(p-1)/p\).  Adding a number in \([0,1/p)\) cannot move any
of these points across the next integer.  Consequently
\[
 \left\lfloor\frac{(n-j)q}{p}+b+\gamma\right\rfloor
 =\left\lfloor\frac{nq}{p}\right\rfloor+b-t.
\]
Thus an arbitrary real shift changes only the integer translate.  This is
also the normalization used in Graham--O'Bryant
\cite[p.~284, footnote~1]{GrahamOBryant2005}.  We may therefore work with
\[
 \B^q_{p,r}:=\left\{\left\lfloor\frac{nq}{p}+r\right\rfloor:n\in\Z\right\},
 \qquad 0<p<q,\quad r\in\Z.
\]
Its indicator is \(q\)-periodic and has exactly \(p\) ones in every block
of \(q\) consecutive positions.  Thus its density is \(p/q\), while its
average spacing---equivalently, its Beatty modulus---is \(q/p\).
Indeed, increasing the Beatty index by \(p\) increases the output by \(q\),
so the set is invariant under translation by \(q\); the indices
\(0,1,\ldots,p-1\) give exactly \(p\) points modulo \(q\).

For each \(i\), write the rational modulus in lowest terms as
\[
 \alpha_i=\frac{b_i}{a_i},\qquad
 a_i,b_i\in\Z_{>0},\qquad 0<a_i<b_i,\qquad (a_i,b_i)=1.
\]
The preceding shift normalization supplies an integer \(r_i\) such that
\(S(\alpha_i,\beta_i)=\B^{b_i}_{a_i,r_i}\).  Choose a positive integer
\(q\) divisible by every \(b_i\), and put
\[
 c_i=\frac q{b_i},\qquad p_i=c_i a_i.
\]
Then
\[
 \B^{b_i}_{a_i,r_i}
 =\left\{\left\lfloor\frac{nb_i}{a_i}+r_i\right\rfloor:n\in\Z\right\}
 =\left\{\left\lfloor\frac{nq}{p_i}+r_i\right\rfloor:n\in\Z\right\}
 =\B^q_{p_i,r_i}.
\]
Thus \((q;(p_i,r_i)_{i=1}^m)\) is a selected common-denominator
presentation of the same partition, with
\[
 \frac1{\alpha_i}=\frac{p_i}{q},\qquad 0<p_i<q.
\]
Each indicator is \(q\)-periodic and has exactly \(p_i\) ones on the
\(q\) residue classes.  Since the indicators sum to one on every residue
class, summing over one period gives
\[
 \sum_{i=1}^m p_i=q.
\]

This presentation is generally noncanonical.  Put
\[
 d=\gcd(q,p_1,\ldots,p_m),\qquad q'=q/d,\qquad p_i'=p_i/d.
\]
Because \(q=c_i b_i\), \(p_i=c_i a_i\), and \((a_i,b_i)=1\),
\[
 \gcd(q,p_i)=c_i.
\]
Hence \(d\mid c_i\) for every \(i\).  Writing \(c_i'=c_i/d\), we have
\[
 q'=c_i'b_i,
 \qquad p_i'=c_i'a_i,
 \qquad
 \B^{q'}_{p_i',r_i}=\B^q_{p_i,r_i}.
\]
Consequently the divided data are still a common-denominator presentation
of exactly the same sets with the same integer shifts.  They satisfy
\[
 0<p_i'<q',\qquad \sum_{i=1}^m p_i'=q',\qquad
 \gcd(q',p_1',\ldots,p_m')=1.
\]
We henceforth rename \(q',p_i'\) as \(q,p_i\), obtaining the
\emph{primitive data}
\begin{equation}\label{eq:primitive}
 0<p_i<q,\qquad \sum_{i=1}^m p_i=q,\qquad
 \gcd(q,p_1,\ldots,p_m)=1.
\end{equation}
The normalized numerators are pairwise distinct: equality \(p_i=p_j\)
would give \(\alpha_i=q/p_i=q/p_j=\alpha_j\).
Theorem~\ref{thm:main} is equivalent to
\[
 q=2^m-1,\qquad
 \{p_1,\ldots,p_m\}=\{1,2,4,\ldots,2^{m-1}\}.
\]
Put \(D=2^m-1\).  Since
\(\alpha_i=q/p_i\), binary primitive data immediately give the moduli
\(D/2^j\), \(0\le j<m\).  Conversely, suppose the moduli are precisely
these numbers.  Their densities are \(2^j/D\).  If \(p_*/q=1/D\) is the
common-denominator representation of the smallest density, then
\(q=Dp_*\).  Writing \(c=p_*\), the numerator representing density
\(2^j/D\) is \(c2^j\).  Hence
\[
 \gcd(q,p_1,\ldots,p_m)
 =c\,\gcd(D,1,2,\ldots,2^{m-1})=c.
\]
Primitivity makes this gcd equal to one, so \(c=1\), proving the converse.

\subsection{Primitive formulation and recurring notation}

In the primitive data \eqref{eq:primitive}, Theorem~\ref{thm:one-third-main}
is equivalent to
\[
 \max_i p_i\ge\frac q3.
\]
The strict negation \(3p_i<q\) will be used in the Fourier argument.  The
same letter \(Q\) always denotes the denominator after passing to one gcd
layer, whereas \(q\) denotes the original common denominator.

Table~\ref{tab:notation} collects the symbols that occur throughout the
proof.

\begin{table}[ht]
\centering
\small
\begin{tabularx}{\textwidth}{@{}p{0.28\textwidth}X@{}}
\toprule
Symbol & Meaning \\
\midrule
\(q,p_i\) & Primitive common-denominator data; the \(i\)-th density is
\(p_i/q\) and its modulus is \(q/p_i\).\\
\(g_i=(p_i,q)\) & Gcd attached to the \(i\)-th numerator.\\
\(g=\min_i g_i\), \(Q=q/g\) & Minimum gcd and the denominator of the
minimum layer.\\
\(A=\{a_i:p_i=ga_i,\ g_i=g\}\) & Distinct unit representatives in the
minimum layer.\\
\(s=\min A\), \(n=|A|-1\) & Smallest representative and the number of
other rows.\\
\(\normQ{x}\) & Least positive cyclic distance from \(x\) to
\(0\pmod Q\).\\
\(K_Q(d)\) & The inverse-sine kernel
\(\sin(\pi/Q)/\sin(\pi d/Q)\).\\
\bottomrule
\end{tabularx}
\caption{Recurring notation.}
\label{tab:notation}
\end{table}

The three bounded cases left by the uniform arguments are verified exactly as
specified in Appendix~\ref{app:verification}.  The two integer verifications
and the rational refinement make no finite decision using floating-point
trigonometry.

\section{Balanced-word deletion}\label{sec:deletion}

The one-third theorem will be useful only if its large component can be
removed without leaving the class of objects needed for induction.  In the
coding word of a Beatty partition, a density at least \(1/3\) corresponds to
a letter of mean spacing at most three.  The purpose of this section is to
show that deleting such a letter preserves balance.  The later reconstruction
in Section~\ref{sec:induction} will use periodicity to recognize the
compressed indicators as rational Beatty sets.

Let \(m\) be a positive integer, let \(\alpha_i>1\) and
\(\beta_i\in\mathbb R\) for \(1\le i\le m\), and suppose that the sets
\(S(\alpha_i,\beta_i)\) partition \(\Z\).  Their coding word is the unique
word \(W\in\{1,\ldots,m\}^{\Z}\) such that, for every \(z\in\Z\) and
every \(i\in\{1,\ldots,m\}\),
\[
 W_z=i\quad\Longleftrightarrow\quad z\in S(\alpha_i,\beta_i).
\]
Let \(\mathcal A\) be a finite alphabet.  For a finite word
\(Y\in\mathcal A^*\) and a letter \(c\in\mathcal A\), let \(|Y|_c\)
denote the number of occurrences of \(c\) in \(Y\).  A bi-infinite word
\(x\in\mathcal A^{\Z}\) is \emph{balanced} if, for every
\(c\in\mathcal A\), any two finite factors \(U,V\) of \(x\) of equal
length satisfy
\[
 \bigl||U|_c-|V|_c\bigr|\le1.
\]

We first verify the standard fact that a Beatty coding word is balanced.
Fix one letter, represented by
\[
 S(r,\beta)=\{\lfloor jr+\beta\rfloor:j\in\Z\},\qquad r>1.
\]
For the integer interval \(I=\{u,u+1,\ldots,u+N-1\}\),
\begin{align*}
 |S(r,\beta)\cap I|
 &=\#\{j\in\Z:u\le\lfloor jr+\beta\rfloor<u+N\}\\
 &=\#\left(\Z\cap
   \left[\frac{u-\beta}{r},\frac{u+N-\beta}{r}\right)\right).
\end{align*}
The last half-open interval has length \(N/r\), so it contains either
\(\lfloor N/r\rfloor\) or \(\lceil N/r\rceil\) integers.  Hence two
factors of length \(N\) contain numbers of this letter differing by at most
one.  Applying this argument to every letter proves balance.

Write \(r_a=\alpha_a\) for the Beatty modulus, or mean spacing, of letter
\(a\).  The following bi-infinite form of Tijdeman's deletion
lemma~\cite[Lemma~3]{Tijdeman2000} makes the endpoint conditions explicit.
Tijdeman proves the broader balanced-word formulation, the rescaling after
deletion, and the structural sharpness of the threshold in
\cite[Lemma~6 and Remarks~2--3]{TijdemanBalanced2000}.

\begin{lemma}[Deletion at rate at most three]\label{lem:deletion}
Let \(W\) be the coding word of a Beatty partition.  If a letter \(a\) has
Beatty rate \(r_a\le3\), then deleting every occurrence of \(a\) and
compressing the remaining positions leaves a balanced word.
\end{lemma}

\begin{proof}
Let \(W'\) be the word obtained from \(W\) by deleting all occurrences of
\(a\).  The increasing sequence of positions occupied by \(a\) has the
form
\[
 s_j=\lfloor jr_a+\theta\rfloor\qquad(j\in\Z)
\]
for a suitable real \(\theta\).  Thus, for every \(h\ge1\),
\begin{equation}\label{eq:beatty-gaps}
 \begin{aligned}
 s_{j+h}-s_j\in\{\lfloor hr_a\rfloor,\lceil hr_a\rceil\}.
 \end{aligned}
\end{equation}
Indeed, \(\lfloor x+y\rfloor-\lfloor x\rfloor\) is always either
\(\lfloor y\rfloor\) or \(\lceil y\rceil\).

\smallskip
\noindent\emph{Minimal lifts.}
Let \(X_1,X_2\) be arbitrary factors of \(W'\) of the same positive
length \(n\).  An occurrence of \(X_i\) corresponds to \(n\) consecutive
retained positions of \(W\).  Let \(V_i\) be the smallest factor of \(W\)
containing those positions.  Its endpoints are retained and therefore are
not \(a\); it contains exactly \(n\) non-\(a\) symbols.  Put
\(k_i=|V_i|_a\).  Then
\begin{equation}\label{eq:minimal-lift}
 |V_i|=n+k_i,
 \qquad |X_i|_b=|V_i|_b\quad(b\ne a).
\end{equation}
After interchanging the factors if necessary, assume \(k_2\ge k_1\) and
put \(d=k_2-k_1\).

\smallskip
\noindent\emph{Comparing deletion counts.}
If \(d=0\), the two lifts have the same length, so balance of \(W\) and
\eqref{eq:minimal-lift} give
\(\bigl||X_2|_b-|X_1|_b\bigr|\le1\) for every \(b\ne a\).

We next rule out \(d\ge2\).  Delete the first symbol and the final
\(d-1\) symbols of \(V_2\), and call the remaining factor \(C\).  Then
\[
 |C|=|V_2|-d=n+k_2-d=n+k_1=|V_1|.
\]
The first deleted symbol and the original last symbol of \(V_2\) are both
non-\(a\).  Therefore at most \(d-2\) of the \(d\) deleted positions are
copies of \(a\), and
\[
 |C|_a\ge k_2-(d-2)=k_1+2.
\]
This contradicts balance for the equal-length factors \(C,V_1\).

\smallskip
\noindent\emph{The exceptional endpoint pattern.}
It remains that \(k_2=k_1+1\).  Write \(k=k_1\) and
\(L=n+k=|V_1|\).  Suppose that a remaining letter \(b\) satisfies
\(\lvert |X_2|_b-|X_1|_b\rvert\ge2\).  Comparing \(V_1\) with either
length-\(L\) factor obtained by deleting one endpoint of \(V_2\) makes the
claim precise.  Write \(V_2=cUe\), where \(c,e\ne a\), and set
\[
 \Delta=|V_2|_b-|V_1|_b,
 \quad \varepsilon_c=\mathbf1_{\{c=b\}},
 \quad \varepsilon_e=\mathbf1_{\{e=b\}}.
\]
Balance gives
\begin{equation}\label{eq:endpoint-comparisons}
 |\Delta-\varepsilon_c|\le1,
 \qquad |\Delta-\varepsilon_e|\le1.
\end{equation}
If \(\Delta\le-2\), either inequality is impossible.  Thus
\(\Delta\ge2\).  Each inequality also gives
\(\Delta\le1+\varepsilon_c\le2\) and
\(\Delta\le1+\varepsilon_e\le2\).  Hence
\[
 \Delta=2,
 \qquad \varepsilon_c=\varepsilon_e=1,
 \qquad V_2=bUb.
\]

Let \(x,y\) be the symbols immediately before and after the chosen
occurrence of \(V_1\).  The factors \(xV_1\) and \(V_1y\) have length
\(L+1=|V_2|\) and at most \(|V_1|_b+1\) copies of \(b\), whereas \(V_2\)
has \(|V_1|_b+2\).  Balance is possible only if \(x=y=b\).  Thus \(W\)
contains the factor \(bV_1b\).

\smallskip
\noindent\emph{Using the rate bound.}
The factor \(bV_1b\) has length \(L+2\), has non-\(a\) endpoints, and
contains exactly \(k\) copies of \(a\).  If \(s_j\) is the last occurrence
of \(a\) before it, then the first occurrence after it is \(s_{j+k+1}\).
All \(L+2\) positions lie strictly between these occurrences, so
\[
 s_{j+k+1}-s_j\ge L+3.
\]
By \eqref{eq:beatty-gaps}, the left side is at most
\(\lceil(k+1)r_a\rceil\); hence
\[
 (k+1)r_a>L+2.
\]
All \(k+1\) copies of \(a\) in \(V_2=bUb\) lie in its interior \(U\),
whose length is \(L-1\).  If \(k\ge1\), the first and last of these copies
are \(k\) \(a\)-steps apart and their positions differ by at most \(L-2\).
Thus \(\lfloor kr_a\rfloor\le L-2\), so \(kr_a<L-1\).  If \(k=0\), the
interior contains the unique copy of \(a\), so \(L-1\ge1\) and the same
strict inequality again holds.  Finally,
\[
 r_a-3=\bigl((k+1)r_a-(L+2)\bigr)
       +\bigl((L-1)-kr_a\bigr)>0.
\]
This contradicts \(r_a\le3\).  No letter \(b\ne a\) can have a discrepancy
of two, and therefore \(W'\) is balanced.
\end{proof}

Thus a component of density at least \(1/3\) can be deleted while preserving
balance.  It remains to produce one.  We now assume the strict opposite,
\(p_i<q/3\) for every \(i\), and derive an impossible Fourier row system.

\section{Fourier reduction to subcritical rigidity}\label{sec:fourier}

Assume that every primitive numerator satisfies \(p_i<q/3\).  At a frequency
coming from the minimum-gcd layer, the explicit rational Beatty transform
receives contributions only from that layer.  The covering identity and the
triangle inequality then force every normalized inverse-sine row to have
mass at least one.  Theorem~\ref{thm:SC} says that the same subcritical data
must contain a deficient row, giving the desired contradiction.

The transform formula and the minimum-gcd mechanism are due to Graham and
O'Bryant~\cite[Theorem~1.1 and Section~4.2, especially
Lemma~4.5]{GrahamOBryant2005}.  Our task here is to isolate the exact finite
row system to which the new rigidity theorem will apply.

For integers \(Q\ge2\) and \(d\) with \(1\le d\le Q/2\), define
\begin{equation}\label{eq:kernel-def}
 K_Q(d)=\frac{\sin(\pi/Q)}{\sin(\pi d/Q)}.
\end{equation}
If \(x\not\equiv0\pmod Q\), let \(r\in\{1,\ldots,Q-1\}\) be its least
positive residue and define
\[
 \normQ{x}=\min\{r,Q-r\}.
\]
Thus \(1\le\normQ{x}\le Q/2\); throughout the paper this is the cyclic
distance of \(x\) from zero.

We first record the covering identity.  Let
\(f_i:\Z/q\Z\to\{0,1\}\) be the indicator
of the \(i\)-th Beatty set in one period, and define
\(\widehat f_i:\Z/q\Z\to\mathbb C\) by
\[
 \widehat f_i(j)=\sum_{t=0}^{q-1}f_i(t)e^{-2\pi\ii jt/q},
 \qquad j\in\{0,\ldots,q-1\},
\]
where \(j,t\) denote the indicated integer representatives of their
residue classes; residue-class brackets are suppressed.
The sets form a partition precisely when
\(\sum_i f_i(t)=1\) for every \(t\in\Z/q\Z\).  Taking Fourier transforms
gives
\[
 \sum_i\widehat f_i(0)=q,
 \qquad
 \sum_i\widehat f_i(j)=0\quad(1\le j<q),
\]
because the Fourier transform of the constant function \(1\) vanishes at
every nonzero frequency.

The Fourier transform of a rational Beatty word was computed explicitly by
Graham and O'Bryant \cite[Theorem~1.1]{GrahamOBryant2005}.  In the present
normalization, its value and support take the following form.  If
\(g_i=(p_i,q)\), \(P_i=p_i/g_i\), and \(Q_i=q/g_i\), choose \(u_i\) by
\[
 1\le u_i<Q_i,\qquad u_i\equiv P_i^{-1}\pmod{Q_i}.
\]
Then
\begin{equation}\label{eq:inverse-lift}
 p_i u_i\equiv g_i\pmod q.
\end{equation}
Indeed, multiply \(P_i u_i=1+tQ_i\) by \(g_i\).
Put \(\omega=e^{2\pi\ii/q}\).  At every nonzero frequency \(j\), the
cited formula is
\begin{equation}\label{eq:fourier-exact}
 \widehat f_i(j)=
 \begin{cases}
 \displaystyle
 g_i\frac{1-\omega^j}{1-\omega^{ju_i}}\,\omega^{-jr_i},
   &g_i\mid j,\\[6pt]
 0,&g_i\nmid j,
 \end{cases}
\end{equation}
where \(r_i\) is the integer shift of the \(i\)-th Beatty set.  If
\(z=\omega^j\), this is equivalently
\begin{equation}\label{eq:fourier-J}
 \widehat f_i(j)=
 \begin{cases}
 \displaystyle\frac{g_i z^{-r_i}}{J_{u_i}(z)},
   &\operatorname{ord}(z)\mid Q_i,\\[6pt]
 0,&\operatorname{ord}(z)\nmid Q_i,
 \end{cases}
 \qquad J_u(z)=1+z+\cdots+z^{u-1}.
\end{equation}
Indeed, \(g_i\mid j\) is equivalent to
\(\operatorname{ord}(z)\mid q/g_i=Q_i\).  In the active case, if
\(d=\operatorname{ord}(z)>1\), the facts \(d\mid Q_i\) and
\((u_i,Q_i)=1\) imply \(z^{u_i}\ne1\), so
\(J_{u_i}(z)=(1-z^{u_i})/(1-z)\ne0\).  Thus the denominator appearing
in the active case of \eqref{eq:fourier-J} is legitimate.

For a primitive rational Beatty partition with pairwise distinct moduli and
data \eqref{eq:primitive}, put
\[
 g_i=(p_i,q),\qquad g=\min_i g_i,\qquad Q=q/g,
\]
and retain the indices with \(g_i=g\).  Write \(p_i=ga_i\) on this layer.
Then the \(a_i\) are distinct units modulo \(Q\).

\begin{proposition}[Minimum-gcd Fourier layer]\label{prop:minimum-layer}
Let \((\B_i)_{i=1}^m\) be a primitive rational Beatty partition with
pairwise distinct moduli and common-denominator data \(p_i,q\).  Suppose that
\(3p_i<q\) for every \(i\).  The minimum layer
\(A=\{a_1,\ldots,a_r\}\subset\U\) satisfies
\begin{equation}\label{eq:subcritical-data}
 3a<Q\quad(a\in A),\qquad \sum_{a\in A}a\le Q.
\end{equation}
Moreover, every \(a_k\in A\) satisfies the necessary row inequality
\begin{equation}\label{eq:row-condition}
 1\le M_k:=\sum_{i\ne k}
 K_Q\!\left(\normQ{a_ka_i^{-1}}\right).
\end{equation}
\end{proposition}

\begin{proof}
The arithmetic assertions follow from
\((p_i/g,q/g)=1\), distinctness of the \(p_i\), and
\[
 g\sum_{a\in A}a=\sum_{g_i=g}p_i\le q=gQ.
\]

\smallskip
\noindent\emph{Isolation of the minimum layer.}
Fix \(k\) in the minimum layer and take the nonzero frequency
\(j=p_k=ga_k\).  The associated root is
\[
 z=e^{2\pi\ii p_k/q}=e^{2\pi\ii a_k/Q}.
\]
Because \((a_k,Q)=1\), the order of \(z\) is exactly \(Q\).  By
\eqref{eq:fourier-J}, the \(i\)-th word is active precisely when this order
divides \(Q_i=q/g_i\), that is, when \(Q\mid Q_i\).  Minimality of \(g\)
gives \(g_i\ge g\), and hence \(Q_i\le Q\).  Therefore
\[
 Q\mid Q_i\text{ and }Q_i\le Q
 \quad\Longleftrightarrow\quad Q_i=Q
 \quad\Longleftrightarrow\quad g_i=g.
\]
Thus only the minimum-gcd layer contributes at this frequency.

\smallskip
\noindent\emph{Normalization of the active coefficients.}
On this layer, let \(u_i\equiv a_i^{-1}\pmod Q\).  The elementary
geometric-sum identity gives
\[
|J_{u_i}(z)|
=\left|\frac{1-z^{u_i}}{1-z}\right|
=\frac{|\sin(\pi a_k u_i/Q)|}{|\sin(\pi a_k/Q)|}.
\]
If \(d_{ki}=\normQ{a_ka_i^{-1}}\), then periodicity and reflection of
the sine function imply
\[
 |\sin(\pi a_ku_i/Q)|=\sin(\pi d_{ki}/Q).
\]
For \(i=k\), the congruence \(a_ku_k\equiv1\pmod Q\) similarly gives
\[
 |J_{u_k}(z)|
 =\frac{\sin(\pi/Q)}{|\sin(\pi a_k/Q)|}.
\]
The common factor \(g\) in \eqref{eq:fourier-J} cancels, and every
translation contributes only a complex number of absolute value one.
Consequently
\[
 \frac{|\widehat f_i(p_k)|}{|\widehat f_k(p_k)|}
 =\frac{|J_{u_k}(z)|}{|J_{u_i}(z)|}
 =\frac{\sin(\pi/Q)}{\sin(\pi d_{ki}/Q)}
 =K_Q(d_{ki}).
\]
Finally, the nonzero-frequency covering identity is
\[
 0=\sum_{i:g_i=g}\widehat f_i(p_k).
\]
Isolating the \(k\)-th term and applying the triangle inequality gives
\[
 |\widehat f_k(p_k)|
 \le\sum_{\substack{i\ne k\\g_i=g}}|\widehat f_i(p_k)|.
\]
The \(k\)-th coefficient is nonzero by its displayed formula, so division
by its absolute value proves \eqref{eq:row-condition}.
\end{proof}

For the rest of the Fourier argument, the \emph{mass} of the row indexed by
\(a\in A\) is
\[
 M_a=\sum_{b\in A\setminus\{a\}}
 K_Q\!\left(\normQ{ab^{-1}}\right).
\]
We say that this row \emph{passes} if \(M_a\ge1\), and that it is
\emph{deficient} if \(M_a<1\).  Proposition~\ref{prop:minimum-layer}
says that every row arising from a hypothetical subcritical Beatty
partition must pass.

The required contradiction is the following dimension-free statement.

\begin{theorem}[Subcritical inverse-sine rigidity]\label{thm:SC}
Let \(Q\ge4\), and let \(A\subset\U\) be a nonempty set represented by
distinct positive integers satisfying \eqref{eq:subcritical-data}.  Then some \(a\in A\)
has
\begin{equation}\label{eq:SC}
 \sum_{b\in A\setminus\{a\}}
 K_Q\!\left(\normQ{ab^{-1}}\right)<1.
\end{equation}
\end{theorem}

This is a subcritical form of the inverse-sine rigidity principle proposed by
Graham and O'Bryant~\cite[Conjecture~5.2]{GrahamOBryant2005}.  Within the
regime \(3a<Q\), it gives the deficient-row conclusion needed here without
the cardinality-dependent hypothesis \(Q>(7/4)^{|A|}\).

\begin{corollary}\label{cor:one-third}
Theorem~\ref{thm:SC} implies Theorem~\ref{thm:one-third-main}.
\end{corollary}

\begin{proof}
By Section~\ref{sec:normalization}, the partition admits primitive data
\eqref{eq:primitive}.  In this notation, the negation of
Theorem~\ref{thm:one-third-main} is \(p_i<q/3\) for every \(i\).
If every \(p_i<q/3\), Proposition~\ref{prop:minimum-layer} constructs a
nonempty set \(A\).  Since every representative satisfies \(a\ge1\) and
\(3a<Q\), necessarily \(Q\ge4\).  Every row of this set is at least one,
whereas
Theorem~\ref{thm:SC} gives a deficient row.  This contradiction proves
\(\max_i p_i\ge q/3\).
\end{proof}

The original shifts and Beatty sequences now leave the argument.  Until
Section~\ref{sec:induction}, we work only with the finite data \((Q,A)\) and
prove that every subcritical pair has a deficient row.

\section{Kernel and minimum-row slack}\label{sec:slack}

The row system is trigonometric and, as stated, has no visible divisor
structure.  We introduce two arithmetic controls.  The kernel is bounded
above by \(1/d\) plus a strict rational error; this one-sided estimate later
supports exact verification.  The \(1/d\)-mass of the smallest row satisfies
an exact slack identity.

The consequence is concrete.  If the smallest row passes, almost all
selected numerator weight lies on divisors of \(Q-s\) or \(Q+s\).  We call
these elements good; the remaining bad weight is small.  Product shells then
turn divisibility and the numerator budget into bounds for other rows.  The
split between \(s>1\) and \(s=1\) is exactly the split between a missing
predecessor that creates a fixed gap and a possible dyadic predecessor chain.

\begin{proposition}[Inverse-sine envelope]\label{prop:kernel}
For integers \(Q\ge2\) and \(d\) with \(1\le d\le Q/2\), put
\(C_0=2\pi-4\).  Then
\begin{equation}\label{eq:kernel-envelope}
 \frac1d\le K_Q(d)
 <\frac1d+C_0\frac d{Q^2}
 <\frac1d+\frac{16d}{7Q^2}.
\end{equation}
\end{proposition}

\begin{proof}
For \(0<x<\pi\), set
\[
 F(x)=\frac{x/\sin x-1}{x^2}.
\]
Euler's product gives, for \(|x|<\pi\),
\[
 \frac{x}{\sin x}
 =\prod_{m=1}^{\infty}
  \left(1-\frac{x^2}{m^2\pi^2}\right)^{-1}
 =1+c_1x^2+c_2x^4+\cdots,
\]
where every coefficient \(c_j\) is positive.  Therefore
\[
 F(x)=c_1+c_2x^2+c_3x^4+\cdots
\]
is strictly increasing for \(0<x<\pi\).  Moreover,
\[
 F(\pi/2)=\frac{\pi/2-1}{\pi^2/4}
          =\frac{2\pi-4}{\pi^2}.
\]
Consequently, for \(0<x\le\pi/2\),
\[
 \frac{x}{\sin x}\le1+F(\pi/2)x^2
 =1+\frac{2\pi-4}{\pi^2}x^2.
\]

For the lower bound, note that \(\sin t/t\) decreases on
\((0,\pi)\).  Indeed,
\[
 \frac{d}{dt}\left(\frac{\sin t}{t}\right)
 =\frac{t\cos t-\sin t}{t^2}<0,
\]
because
\[
 \sin t-t\cos t=\int_0^t u\sin u\,du>0.
\]
Take \(x=\pi d/Q\) and \(y=\pi/Q\).  Since \(y\le x\), monotonicity
gives \(\sin y/y\ge\sin x/x\), and hence
\[
 K_Q(d)=\frac{\sin y}{\sin x}\ge\frac1d.
\]
For the upper bound, use
\[
 K_Q(d)=\frac{\sin y}{y}\frac1d\frac{x}{\sin x}
\]
and \(\sin y/y<1\) to obtain
\[
 K_Q(d)<\frac1d\left(1+(2\pi-4)\frac{d^2}{Q^2}\right).
\]
The rational comparison
\(2\pi-4<16/7\) follows from \(\pi<22/7\).
\end{proof}

The useful row is the one indexed by the smallest element.  Let
\(s=\min A\), put \(N=Q-s\), and, for each
\(a\in A\setminus\{s\}\), define
\[
 d_a=\normQ{sa^{-1}},\qquad \delta_a=d_aa-N.
\]

\begin{lemma}[Minimum-row slack identity]\label{lem:slack}
For every \(a\in A\setminus\{s\}\),
\(d_aa\ge N\),
\begin{equation}\label{eq:slack}
 1-\sum_{a\ne s}\frac1{d_a}
 =\frac{Q-\sum_{a\in A}a}{N}
 +\sum_{a\ne s}\frac{\delta_a}{d_aN},
\end{equation}
and
\begin{equation}\label{eq:discrete-gap}
 \delta_a\in\{jQ:j\ge0\}\cup\{2s+jQ:j\ge0\}
 \subset\{0,2s\}\cup[Q,\infty).
\end{equation}
In particular, \(\sum_{a\ne s}d_a^{-1}\le1\).
\end{lemma}

\begin{proof}
By the definition of \(d_a\), one of
\[
 d_a\equiv sa^{-1}\pmod Q,
 \qquad
 d_a\equiv-sa^{-1}\pmod Q
\]
holds.  Multiplication by \(a\) gives
\(d_aa\equiv s\pmod Q\) or \(d_aa\equiv-s\pmod Q\).  In the second
case positivity gives
\[
 d_aa=Q-s+jQ=N+jQ\qquad(j\ge0).
\]
In the first case, \(d_aa=s\) is impossible: it would imply
\(d_a a=s<a\), whereas \(d_a\ge1\).  The next positive representative
is therefore
\[
 d_aa=s+(j+1)Q=N+2s+jQ\qquad(j\ge0).
\]
This proves both \(d_aa\ge N\) and \eqref{eq:discrete-gap}.  Since
\(d_aa=N+\delta_a\),
\[
 \frac1{d_a}=\frac aN-\frac{d_aa-N}{d_aN}.
\]
Summing over \(a\ne s\) gives
\[
 \sum_{a\ne s}\frac1{d_a}
 =\frac{\sum_{a\in A}a-s}{N}
  -\sum_{a\ne s}\frac{\delta_a}{d_aN}.
\]
Using \(N=Q-s\) and rearranging proves \eqref{eq:slack}.  Every term on
its right is nonnegative because \(\sum A\le Q\) and
\(\delta_a\ge0\).
\end{proof}

The sharp kernel envelope converts a passing minimum row into a small exact
slack.

\begin{lemma}[Passing-row slack]\label{lem:passing-slack}
Let \(n=|A|-1\).  If the \(s\)-row has mass at least one, then
\begin{equation}\label{eq:E-bound}
 E:=N\left(1-\sum_{a\ne s}\frac1{d_a}\right)
 <\frac{8n}{7}.
\end{equation}
Call each \(a\in A\setminus\{s\}\) \emph{good} if and only if
\(a\mid Q-s\) or \(a\mid Q+s\), and
\emph{bad} otherwise.
If \(W\) is the sum of these bad elements (with the empty sum understood as
zero), then
\begin{equation}\label{eq:bad-weight-general}
 W<\frac{16n}{7}.
\end{equation}
\end{lemma}

\begin{proof}
By Proposition~\ref{prop:kernel}, a passing row gives
\[
 1<\sum_{a\ne s}\frac1{d_a}
 +\frac{C_0}{Q^2}\sum_{a\ne s}d_a.
\]
Since every \(d_a\le Q/2\),
\[
 E<\frac{C_0N}{Q^2}\sum_{a\ne s}d_a
 \le\frac{C_0nN}{2Q}<\frac{C_0n}{2}<\frac{8n}{7},
\]
which proves \eqref{eq:E-bound}.

\smallskip
\noindent\emph{The divisor interpretation.}
An element \(a\) is good precisely when
\(\delta_a\in\{0,2s\}\).  If \(\delta_a=0\), then \(d_aa=Q-s\), so
\(a\mid Q-s\).  If
\(\delta_a=2s\), then \(d_aa=Q+s\), so \(a\mid Q+s\).

Conversely, suppose first that \(a\mid Q-s\).  Then
\[
 d=\frac{Q-s}{a}
\]
satisfies \(da\equiv-s\pmod Q\).  Since \(a>s\), we have \(a\ge2\)
and \(1\le d<Q/2\).  Thus \(d\) is the least cyclic representative,
so \(d=d_a\) and \(\delta_a=0\).  Suppose next that
\(a\mid Q+s\).  If \(a\ge3\), then \(3s<Q\) gives
\[
 \frac{Q+s}{a}\le\frac{Q+s}{3}<\frac{4Q}{9}<\frac Q2.
\]
Again the quotient is the least cyclic representative, now giving
\(\delta_a=2s\).  If \(a=2\), then necessarily \(s=1\).  Since \(a\)
is a unit modulo \(Q\), the integer \(Q\) is odd, so
\(2\mid Q-1=Q-s\); the preceding \(Q-s\) argument applies.
Therefore goodness is exactly the condition
\(\delta_a\in\{0,2s\}\).  By \eqref{eq:discrete-gap}, every bad element
has \(\delta_a\ge Q\).

\smallskip
\noindent\emph{The bad-weight bound.}
Multiplying \eqref{eq:slack} by \(N\) makes the location of every term
explicit:
\[
 E=Q-\sum_{a\in A}a+\sum_{a\ne s}\frac{\delta_a}{d_a}.
\]
All terms on the right are nonnegative.  For a bad element,
\[
 \frac{\delta_a}{d_a}
 =\frac{\delta_aa}{N+\delta_a}
 \ge\frac{Qa}{2Q-s}>\frac a2.
\]
If \(W>0\), summing this strict inequality over the nonempty bad set, and
using the nonnegativity of every other term in \(E\), gives
\(E>W/2\), and hence \(W<2E<16n/7\).  If \(W=0\), the passing-row
hypothesis forces \(n\ge1\) (a row with no other vertex has mass zero), so
\(W=0<16n/7\) directly.  This proves \eqref{eq:bad-weight-general} in all
cases.
\end{proof}

We will repeatedly group the other elements by the possible values of the
integer product \(bd_{ab}\).  The following bookkeeping lemma converts those
product shells and the numerator budget into reciprocal-row bounds.

\begin{lemma}[Product-shell bookkeeping]\label{lem:shell-bookkeeping}
For distinct \(a,b\in A\), write
\[
 d_{ab}=\normQ{ab^{-1}},\qquad P_{ab}=bd_{ab}.
\]
Then
\[
 P_{ab}\in\{a,Q-a,Q+a,2Q-a,2Q+a,\ldots\},
 \qquad \frac1{d_{ab}}=\frac b{P_{ab}}.
\]
For every \(a\in A\), every subset \(B\subseteq A\setminus\{a\}\),
and all real numbers \(L\ge0\) and \(T>0\), if
\[
 \sum_{b\in B}b\le L
 \qquad\text{and}\qquad
 P_{ab}\ge T\quad(b\in B),
\]
then
\[
 \sum_{b\in B}\frac1{d_{ab}}\le\frac LT.
\]
Also, \(P_{ab}=a\) implies \(b\mid a\).

In the special case \(s=1\), call \(b\) good when
\(b\mid Q-1\) or \(b\mid Q+1\).  If such a good \(b\) satisfies
\(P_{ab}\in\{Q-a,Q+a\}\), then \(b\mid a-1\) or \(b\mid a+1\).
If moreover \(a>1\) and \(b>a\), this forces \(b=a+1\).
\end{lemma}

\begin{proof}
The definition of cyclic distance gives \(d_{ab}b\equiv\pm a\pmod Q\).
Since \(P_{ab}=d_{ab}b\) is positive, it belongs to the displayed list.
For the weight estimate, fix \(a\in A\),
\(B\subseteq A\setminus\{a\}\), \(L\ge0\), and \(T>0\) as in the
statement.  For each \(b\in B\), positivity and \(P_{ab}\ge T\) give
\[
 \frac1{d_{ab}}=\frac b{P_{ab}}\le\frac bT.
\]
Summing over \(B\) yields
\[
 \sum_{b\in B}\frac1{d_{ab}}
 \le\frac1T\sum_{b\in B}b\le\frac LT.
\]
Also, \(P_{ab}=a\) implies \(b\mid a\), because \(P_{ab}=bd_{ab}\).

For the final assertion, goodness gives \(b\mid Q+\varepsilon\) for some
\(\varepsilon\in\{-1,1\}\), while first-shell membership gives
\(b\mid Q+\delta a\) for some \(\delta\in\{-1,1\}\).  Subtracting
shows \(b\mid\delta a-\varepsilon\), hence \(b\mid a-1\) or
\(b\mid a+1\).  If in addition \(a>1\) and \(b>a\), then
\(0<a-1<b\), so \(b\nmid a-1\).  Hence \(b\mid a+1\).  Since
\(0<a+1\le b\), the positive multiple \(a+1\) of \(b\) must equal
\(b\).
\end{proof}

The common reduction is complete.  We first use it when \(s>1\), where the
least good element has no selected divisor predecessor.  The exceptional
anchor \(s=1\) is treated afterward.

\section{Noncritical minimum rows: \texorpdfstring{\(s>1\)}{s > 1}}\label{sec:s-positive}

Assume for contradiction that every row passes.  The minimum-row estimate
guarantees a good element; choose the least one, \(x\).  Because \(s>1\),
no selected proper divisor of \(x\) can precede it.  The resulting empty
product shell creates a fixed gap in the \(x\)-row.  Comparing that gap with
the bad-weight and kernel-error bounds first reduces the problem to a small
box; an exact upward-rounded knapsack exclusion then finishes the case.

\subsection{Analytic reduction to \texorpdfstring{\(Q\le301\)}{Q <= 301}}

\begin{proposition}[Analytic cutoff for \(s>1\)]
\label{prop:s-positive-cutoff}
Under the hypotheses of Theorem~\ref{thm:SC}, suppose that \(s=\min A>1\)
and every row passes.  Then
\[
 7\le Q\le301.
\]
\end{proposition}

\begin{proof}
Suppose that every row passes, so for every \(a\in A\),
\[
 M_a:=\sum_{b\in A\setminus\{a\}}
 K_Q\!\left(\normQ{ab^{-1}}\right)\ge1.
\]
Write \(r=|A|\), \(n=r-1\), and retain the notation of
Lemma~\ref{lem:passing-slack}.  Necessarily \(n\ge1\), since otherwise the
only row has mass zero.  There is a good element: otherwise each of
the \(n\) elements above \(s\ge2\) is bad, so \(W\ge3n>16n/7\).
Let \(x\) be the least good element.

\smallskip
\noindent\emph{The missing predecessor shell.}
No selected proper divisor of \(x\) exists.  Such a divisor would inherit
divisibility of \(Q-s\) or \(Q+s\) and would be a smaller good element.
If the possible divisor were \(s\), then \(s\mid x\) together with
\(x\mid Q\pm s\) would give \(s\mid Q\), contradicting \((s,Q)=1\) and
\(s>1\).
For \(a\in A\setminus\{x\}\), put
\[
 d_a^{(x)}=\normQ{xa^{-1}},\qquad P_a=ad_a^{(x)}.
\]
By Lemma~\ref{lem:shell-bookkeeping}, and because \(0<x<Q/3\), the
possible products occur in the increasing shell list
\begin{equation}\label{eq:s-positive-shell-list}
 x,\quad Q-x,\quad Q+x,\quad 2Q-x,\quad 2Q+x,\ldots.
\end{equation}
The first possibility would imply \(a\mid x\), which the preceding
paragraph excludes.  Hence
\begin{equation}\label{eq:s-positive-first-shell}
 P_a\ge Q-x\qquad(a\ne x).
\end{equation}

For a good \(a\ne x\), suppose that \(P_a\) is one of
\(Q-x,Q+x\).  Goodness says that \(a\) divides one of
\(Q-s,Q+s\), while this first-shell assumption says that it divides one
of \(Q-x,Q+x\).  Subtracting the corresponding multiples shows
\(a\mid x-s\) or \(a\mid x+s\).  The first is impossible because
\(a>x>x-s>0\).  Since \(a>x>s\), we also have \(0<x+s<2a\), so the
second divisibility forces \(a=x+s\).  Thus at most this single good
element can lie in the first shell; every other good element satisfies
\begin{equation}\label{eq:s-positive-far-shell}
 P_a\ge2Q-x.
\end{equation}
Define
\[
 c=\begin{cases}
 x+s,&\text{if }x+s\in A\text{ is good},\\
 0,&\text{otherwise}.
 \end{cases}
\]
This keeps \(c\) disjoint from the bad weight \(W\).  If
\(S=\sum_{a\in A}a\), then \(1/d_a^{(x)}=a/P_a\).  Assigning
\(s\), every bad element, and the possible exceptional element \(c\) to
the smaller denominator \(Q-x\), and all remaining good elements to
\(2Q-x\), can only increase the reciprocal sum.  Therefore
\begin{equation}\label{eq:s-positive-H}
 H_x:=\sum_{a\ne x}\frac1{d_a^{(x)}}
 \le\frac{s+W+c}{Q-x}
 +\frac{S-s-W-x-c}{2Q-x}.
\end{equation}

\smallskip
\noindent\emph{A uniform gap in the \(x\)-row.}
Put
\[
 X=\frac xQ,\qquad w=\frac WQ,\qquad
 \alpha=1-\frac{x+s+c}{Q},\qquad D=(1-X)(2-X).
\]
Using \(S\le Q\) in \eqref{eq:s-positive-H}, and observing that
\((s+c)/Q=1-X-\alpha\), gives
\begin{align*}
 H_x
 &\le \frac{(s+c)/Q+w}{1-X}
      +\frac{1-(s+c+x)/Q-w}{2-X}\\
 &=1-\frac{\alpha-w}{(1-X)(2-X)}.
\end{align*}
Thus
\begin{equation}\label{eq:s-positive-gap-a0}
 1-H_x\ge\frac{\alpha-w}{D}.
\end{equation}
If \(c=0\), then \(x+s<2Q/3\).  If \(c=x+s\), subcriticality of
\(c\) gives \(x+s+c=2c<2Q/3\).  Hence \(\alpha>1/3\).  Also,
\(x+s+c+W\le S\le Q\), so
\(0\le w\le\alpha\), and \(D<2\) because \(0<X<1/3\).

Consider the affine function
\[
 F(w)=\frac{\alpha-w}{D}-\left(\frac16-\frac w2\right)
\]
on \([0,\alpha]\).  At its endpoints,
\[
 F(0)=\frac\alpha D-\frac16>\frac\alpha2-\frac16>0,
 \qquad
 F(\alpha)=\frac\alpha2-\frac16>0.
\]
Every value of an affine function on an interval is a convex combination of
its endpoint values, so \(F(w)>0\) throughout.  Therefore
Lemma~\ref{lem:passing-slack} gives
\[
 1-H_x>\frac16-\frac{W}{2Q}
 >\frac16-\frac{8n}{7Q}.
\]
The full sine-kernel remainder is also explicit: Proposition~\ref{prop:kernel}
and \(d_a^{(x)}\le Q/2\) give
\[
 M_x-H_x<\frac{16}{7Q^2}\sum_{a\ne x}d_a^{(x)}
 \le\frac{16}{7Q^2}\,n\frac Q2=\frac{8n}{7Q}.
\]
Thus
\begin{equation}\label{eq:s-positive-total-gap}
 1-M_x>\frac16-\frac{16n}{7Q}.
\end{equation}
A passing row forces \(Q<96n/7\).

\smallskip
\noindent\emph{The denominator cutoff.}
Since \(A\) consists of \(r\) distinct integers
at least two,
\[
 Q\ge\sum_{a\in A}a\ge2+3+\cdots+(r+1)=\frac{r(r+3)}2.
\]
Combining the two bounds gives
\[
 7r^2-171r+192<0.
\]
For \(r\ge24\), the left side is increasing and its value at \(r=24\)
is \(120>0\).  Hence \(r\le23\), and then
\[
 Q<\frac{96\cdot22}{7}=301+\frac57,
\]
so the integer \(Q\) satisfies \(Q\le301\).  Since \(s\ge2\) and
\(3s<Q\), also \(Q\ge7\).
\end{proof}

\subsection{Exact finite exclusion}

For \(7\le Q\le301\), the necessary row conditions admit upward-rounded
integer bounds.  Exact \(0\)--\(1\) knapsack bounds show that the prescribed
minimum cannot belong to any all-passing set.

\begin{proposition}[Exact finite exclusion for \(s>1\)]
\label{prop:s-positive-finite}
For \(7\le Q\le301\) and \(2\le s<Q/3\) with \((s,Q)=1\), there is no
set
\[
 A\subseteq\{a\in\Z:s\le a<Q/3,\ (a,Q)=1\}
\]
with \(s\in A\), \(\sum_{a\in A}a\le Q\), and every row passing.
\end{proposition}

\begin{proof}
Appendix~\ref{app:s-positive-verifier} defines an exact knapsack score and
a synchronous peeling procedure for every admissible pair \((Q,s)\).  Its
strict upper envelope in Proposition~\ref{prop:kernel} implies that, after
multiplication by \(10^7\) and upward rounding, every member of a genuine
all-passing set has knapsack score strictly larger than \(10^7\).  Thus no
peeling round can remove such a member.  The exact verification in
Appendix~\ref{app:s-positive-verifier} checks all admissible pairs and removes
the prescribed minimum \(s\) in each case.  Hence no all-passing set exists.
\end{proof}

The analytic cutoff and the exact exclusion have matching ranges, and hence
settle the noncritical case of the rigidity theorem.

\begin{theorem}[The \(s>1\) case]\label{thm:s-positive}
Under the hypotheses of Theorem~\ref{thm:SC}, if \(s=\min A>1\), then
some row is deficient.
\end{theorem}

\begin{proof}
If every row passed, Proposition~\ref{prop:s-positive-cutoff} would give
\(7\le Q\le301\), contrary to Proposition~\ref{prop:s-positive-finite}.
\end{proof}

The fixed-gap argument used \(s>1\) exactly to exclude a selected proper
divisor of the least good element.  When \(s=1\), the anchor is an unavoidable
predecessor, and the dyadic chains it permits require a separate argument.

\section{The critical unit row: \texorpdfstring{\(s=1\)}{s = 1}}\label{sec:s-one}

The fixed-gap argument fails because the anchor \(1\) divides every good
element.  We therefore organize the proof around the least selected good
element \(x\).  If \(x\ge3\), or if \(x=2\) and \(3\in A\), a direct shell
gap remains.  If \(x=2\), \(3\notin A\), and all selected good elements are
dyadic, their geometric total is too small.  The only genuinely critical
configuration has a least nondyadic good element \(y\); information from its
dyadic predecessors is then propagated through a harmonic envelope.

The four alternatives are exhaustive.  Analytic estimates bound the first
three directly and give a preliminary bound for the fourth; an exact rational
refinement then brings every branch into \(Q\le715\).  A single exact global
verifier excludes the resulting all-passing selection system.

\subsection{Minimum-row consequences and the least good element}

We begin by specializing the slack estimates of
Section~\ref{sec:slack} to the anchor \(1\).  A passing minimum row controls
both the unused numerator budget and the total bad weight, and it forces the
existence of a good element.  The least such element is the first branching
parameter of the proof.

Throughout this section,
\[
 A\subset\U,\qquad \min A=1,\qquad
 3a<Q\ (a\in A),\qquad \sum_{a\in A}a\le Q,
\]
and \(n=|A|-1\).  An element \(a>1\) is \emph{good} if
\(a\mid Q-1\) or \(a\mid Q+1\), and is \emph{bad} otherwise.  Let
\(W\) be the sum of the bad elements.

\begin{lemma}[The minimum row]\label{lem:s1-minimum-row}
If the row at \(1\) passes, then
\begin{equation}\label{eq:s1-RW}
 R:=Q-\sum_{a\in A}a<\frac{8n}{7},
 \qquad W<\frac{16n}{7}.
\end{equation}
If every row passes, at least one element above \(1\) is good.
\end{lemma}

\begin{proof}
Multiplying \eqref{eq:slack} by \(Q-1\) gives the exact identity
\[
 E=(Q-1)\left(1-\sum_{a\ne1}\frac1{d_a}\right)
   =R+\sum_{a\ne1}\frac{\delta_a}{d_a}.
\]
All terms in the last sum are nonnegative, so
\(0\le R\le E<8n/7\) by Lemma~\ref{lem:passing-slack}; its bad-weight
estimate gives \(W<16n/7\).
If there were no good element, then
\[
 W\ge2+3+\cdots+(n+1)=\frac{n(n+3)}2.
\]
For \(n\ge2\), the difference between this lower bound and \(16n/7\) is
\[
 \frac{n(n+3)}2-\frac{16n}{7}
 =\frac{n(7n-11)}{14}>0,
\]
contradicting \eqref{eq:s1-RW}.  If \(n=0\), the
minimum row has mass zero.  If \(n=1\), the unique element cannot be the
bad element \(2\), because \((2,Q)=1\) makes \(Q\) odd and hence
\(2\mid Q\pm1\); if it is at least three, its weight already exceeds
\(16/7\).  Thus the all-passing case has a good element.
\end{proof}

\subsection{Directly bounded branches}

The first two configurations, \(x\ge3\) and \(x=2\) with \(3\in A\), retain
a direct shell gap.  When \(x=2\), \(3\notin A\), and every selected good
element is dyadic, a geometric total estimate gives the required bound.  Only
the remaining configuration admits a persistent nondyadic predecessor chain.

Let \(x\) be the least selected good element.

\begin{lemma}[Least good element at least three]\label{lem:s1-x-ge-three}
If \(x\ge3\) and every row passes, then
\[
 n\le36,\qquad Q\le703.
\]
\end{lemma}

\begin{proof}
\noindent\emph{The first two product shells.}
No bad element divides \(x\): a divisor of a number dividing \(Q-1\) or
\(Q+1\) is itself good.  In the \(x\)-row, the anchor \(b=1\) has
product \(P_{x1}=x\).  A bad element cannot have product \(x\), so by the
shell list its product is at least \(Q-x\).  Lemma~\ref{lem:shell-bookkeeping}
also shows that if another good \(b>x\) has product \(Q-x\) or \(Q+x\),
then \(b=x+1\).  Every other good product is at least \(2Q-x\).

\smallskip
\noindent\emph{The reciprocal baseline.}
Let
\[
 c=\begin{cases}x+1,&x+1\in A\text{ and is good},\\0,&\text{otherwise},
 \end{cases}
\]
and let \(S=\sum A\).  The total numerator weight other than \(1,x\) is
\(S-1-x\le Q-x-1\).  Put \(L=W+c\).  Placing the bad elements and the
possible \(c\) in the first shell, and all remaining weight in the farther
shell, gives
\begin{equation}\label{eq:s1-x-H}
 H_x:=\sum_{b\ne x}\frac1{\normQ{xb^{-1}}}
 \le F(L):=\frac1x+\frac{Q-x-1-L}{2Q-x}+\frac{L}{Q-x}.
\end{equation}
The coefficient of \(L\) is
\(1/(Q-x)-1/(2Q-x)>0\), so \(F\) is increasing.  Since
\(c\le x+1\), we have \(L\le W+x+1\), and hence
\[
 H_x\le F(W+x+1)
 =F(x+1)+W\left(\frac1{Q-x}-\frac1{2Q-x}\right).
\]
The baseline value \(F(x+1)\) is
\[
 B(Q,x)=\frac1x+\frac{Q-2x-2}{2Q-x}+\frac{x+1}{Q-x},
\]
and for \(x\ge3\), \(Q>3x\),
\begin{equation}\label{eq:B-bound}
 B(Q,x)\le\frac56+\frac2Q.
\end{equation}
Indeed, write \(Q=3x+t\), where the integer \(t\ge1\).  Direct expansion
gives
\[
 6Qx(Q-x)(2Q-x)\left(\frac56+\frac2Q-B(Q,x)\right)={}
 x^3(24x-78)+tx^2(53x-120)
 +t^2x(27x-66)+t^3(4x-12).
\]
For \(x\ge4\), the four coefficient factors
\(24x-78,53x-120,27x-66,4x-12\) are at least
\(18,92,42,4\), respectively, so every term is positive.  For \(x=3\)
the expression is \(-162+351t+45t^2\), which is at least \(234>0\)
because \(t\ge1\).  This proves \eqref{eq:B-bound}.

Thus
\[
 H_x\le B(Q,x)+W\frac{Q}{(Q-x)(2Q-x)}.
\]

\smallskip
\noindent\emph{The cutoff.}
To check the promotion coefficient, put \(u=x/Q<1/3\).  Then
\((1-u)(2-u)>(2/3)(5/3)=10/9\), and therefore
\[
 \frac{Q}{(Q-x)(2Q-x)}
 =\frac1{Q(1-u)(2-u)}<\frac9{10Q}.
\]
Using \(W<16n/7\), and adding the kernel remainder
\[
 M_x-H_x<\frac{16}{7Q^2}\sum_{b\ne x}\normQ{xb^{-1}}
 \le\frac{8n}{7Q},
\]
we obtain
\begin{equation}\label{eq:s1-x-M}
 M_x<\frac56+\frac{2+16n/5}{Q}.
\end{equation}
A passing row \(M_x\ge1\) forces \(Q<12+96n/5\).  Distinctness gives
\begin{equation}\label{eq:s1-distinct-sum}
 Q\ge1+2+\cdots+(n+1)=\frac{(n+1)(n+2)}2.
\end{equation}
At \(n=37\), the lower bound is \(741\), whereas the strict upper bound
is \(12+96\cdot37/5=722.4\).  Thereafter the lower bound increases by
\(n+2\ge39\) when \(n\) is increased by one, while the upper bound
increases only by \(96/5\).  Thus \(n\le36\).  Substitution gives
\(Q<703.2\), and integrality yields \(Q\le703\).
\end{proof}

We may now assume \(x=2\), so \(Q\) is odd.

\begin{lemma}[The anchor at three]\label{lem:s1-three}
If \(2,3\in A\) and every row passes, then
\[
 n\le36,\qquad Q\le715.
\]
\end{lemma}

\begin{proof}
\noindent\emph{The first shell of the \(3\)-row.}
The element \(2\) is good because \(Q\) is odd.  Also, oddness and
\((3,Q)=1\) imply \(Q\equiv1\) or \(5\pmod6\), so \(3\mid Q-1\) or
\(3\mid Q+1\); thus \(3\) is good.  In the \(3\)-row, the anchor has
product \(P_{31}=3\).  Since \(Q\) is odd, \((Q-3)/2\) is a positive
integer below \(Q/2\), and
\(2\cdot(Q-3)/2\equiv-3\pmod Q\).  Thus
\(P_{32}=Q-3\), in the first shell.  By
Lemma~\ref{lem:shell-bookkeeping}, a good \(b>3\) in that shell satisfies
\(b\mid2\) or \(b\mid4\), so only \(b=4\) is possible.

\smallskip
\noindent\emph{The row estimate.}
Because \(Q\) is odd, \(Q\equiv\pm1\pmod4\); thus \(4\) is good whenever
it is selected, and it is not included in the bad weight \(W\).  Let
\(c=4\) if \(4\in A\), and \(c=0\) otherwise.  The total numerator
weight other than \(1,3\) is at most \(Q-4\).  Put
\[
 F_3(L)=\frac13+\frac{Q-4-L}{2Q-3}+\frac{L}{Q-3}.
\]
The function is increasing.  Bad elements, \(b=2\), and the possible
\(b=4\) have total weight \(W+2+c\le W+6\), while every other good
element has product at least \(2Q-3\).  Therefore
\[
 H_3\le F_3(W+2+c)\le F_3(W+6)
 =F_3(6)+W\left(\frac1{Q-3}-\frac1{2Q-3}\right).
\]
Since \(Q>3\cdot3\) and \(Q\) is odd, automatically \(Q\ge11\).  At
\(L=6\),
\[
 \frac13+\frac{Q-10}{2Q-3}+\frac6{Q-3}
 \le\frac56+\frac4Q;
\]
indeed,
\[
 6Q(Q-3)(2Q-3)\left[
 \frac56+\frac4Q-
 \left(\frac13+\frac{Q-10}{2Q-3}+\frac6{Q-3}\right)\right]
 =3(9Q^2-87Q+72)>0.
\]
The last quadratic is positive for \(Q\ge11\): its derivative
\(18Q-87\) is positive there, and its value at \(Q=11\) is \(204\).
The remaining argument converts this reciprocal estimate to the true kernel
mass and then compares the linear upper bound with the distinctness lower
bound.
For the remaining bad weight, the same promotion
coefficient as in Lemma~\ref{lem:s1-x-ge-three}, now with \(x=3\), is less
than \(9/(10Q)\).  Together with \(W<16n/7\) and the kernel correction
\(8n/(7Q)\), this gives
\[
 M_3<\frac56+\frac{4+16n/5}{Q}.
\]
Thus \(Q<24+96n/5\).  At \(n=37\), the distinctness lower bound is
\(741\), whereas this upper bound is \(734.4\); thereafter the lower
bound grows faster.  Hence \(n\le36\), and then \(Q<715.2\), so
\(Q\le715\).
\end{proof}

\subsection{The critical dyadic branch}

It remains to treat
\begin{equation}\label{eq:critical-branch}
 2\in A,\qquad 3\notin A.
\end{equation}
Here and below, \emph{dyadic} means a positive power of two.
If all selected good elements are dyadic, their geometric total
is too small to support a passing system once the bad-weight bound is
included.  Otherwise the least nondyadic good element \(y\) has only dyadic
selected good predecessors.

\begin{lemma}[No nondyadic good element]\label{lem:no-nondyadic}
Assume \eqref{eq:critical-branch}.  If every selected good element is a
power of two and every row passes, then
\[
 n\le17,\qquad Q\le171.
\]
\end{lemma}

\begin{proof}
The largest selected good dyadic element \(L=2^k\) exists because \(2\)
is selected and good.  Let \(G\) be the sum of the anchor \(1\) and all
selected good elements.  The selected good elements form a subset of
\(\{2,4,\ldots,L\}\), so the geometric sum gives
\[
 G\le1+(2+4+\cdots+L)=2L-1.
\]
Subcriticality gives \(L<Q/3\), and hence \(G<2Q/3-1\).  The set \(A\)
is the disjoint union of the anchor, the good elements, and the bad
elements.  Consequently \(\sum A=G+W\) and, by the definition of \(R\),
the exact identity is \(Q=G+W+R\).  Lemma~\ref{lem:s1-minimum-row} now gives
\[
 Q=G+W+R<\frac{2Q}{3}-1+\frac{16n}{7}+\frac{8n}{7},
\]
so \(Q<72n/7-3\).  If \(n=18\), the distinctness lower bound is
\(190\), while this upper bound is \(1296/7-3<183\).  On increasing
\(n\), the lower bound grows by \(n+2\ge20\), whereas the upper bound
grows by only \(72/7\).  Therefore \(n\le17\).  Substitution yields
\(Q<1224/7-3<172\), hence \(Q\le171\).
\end{proof}

\subsection{Harmonic propagation from the least nondyadic element}

Suppose henceforth that the row at \(1\) passes and that a selected
nondyadic good element exists, and let \(y\) be the least one.  Choose
\(v\) maximal such that \(2^v\mid y\), and write
\begin{equation}\label{eq:y-factorization}
 y=2^vm,\qquad m\ge3\text{ odd},\qquad
 J=\lfloor\log_2(m-1)\rfloor.
\end{equation}
Thus \(2^v\Vert y\) and \(m=y/2^v\) is odd.
In the critical branch \eqref{eq:critical-branch}, one has \(y\ge5\).
Write \(\Harm_j=\sum_{k=1}^j1/k\).

The next estimate separates the good elements according to their position
relative to \(y\).  Proper dyadic divisors contribute an exact geometric
sum; smaller nondivisors are few; elements between \(y\) and \(2y\) are
controlled by two complementary harmonic estimates; and elements at least
\(2y\) admit a uniform far-shell baseline.

\begin{lemma}[Dyadic--close harmonic envelope]\label{lem:dyadic-close}
Let
\[
 H_y^{\rm good}:=\frac1y+
 \sum_{\substack{b\in A\setminus\{1,y\}\\ b\text{ is good}}}
 \frac1{\normQ{yb^{-1}}};
\]
thus the superscript includes the anchor \(b=1\) as well as every selected
good element other than \(y\).  This contribution to the \(y\)-row satisfies
\begin{equation}\label{eq:dyadic-close}
 H_y^{\rm good}
 <\frac2m+\frac{10J}{3y}+\frac{2\Harm_{y-1}}{y-1}.
\end{equation}
\end{lemma}

\begin{proof}
Because \(y\) is good, it divides one of \(Q-1,Q+1\).  Thus choose
\(\eta\in\{\pm1\}\) and an integer \(D\) with \(Q=Dy+\eta\).  Since
\(y<Q/3\), necessarily \(D\ge3\).  Put
\[
 \mathcal S_y=\{1\}\cup
 \{b\in A\setminus\{1,y\}:b\text{ is good}\},
\]
and partition \(\mathcal S_y\) into
\[
 \begin{aligned}
 P&=\{b\in\mathcal S_y:b<y,\ b\mid y\},&
 X&=\{b\in\mathcal S_y:b<y,\ b\nmid y\},\\
 \mathcal C&=\{b\in\mathcal S_y:y<b<2y\},&
 F&=\{b\in\mathcal S_y:b\ge2y\}.
 \end{aligned}
\]
By the choice of \(y\), every element of \(\mathcal S_y\) below \(y\) is
dyadic.  If \(p_0=\sum_{b\in P}b\), then
\[
 p_0\le1+2+4+\cdots+2^v=2^{v+1}-1=\frac{2y}{m}-1.
\]
For \(b\in P\), the integer \(y/b<Q/2\) satisfies
\((y/b)b=y\), so it is the least cyclic representative:
\(d_{yb}=y/b\).  Thus the \(P\)-contribution is exactly \(p_0/y\).

\smallskip
\noindent\emph{Dyadic divisors and the far shell.}
For \(b\in F\), goodness permits a sign \(\varepsilon\in\{\pm1\}\)
and a positive integer \(L\) with
\(Lb=Q+\varepsilon\).  Then
\(b^{-1}\equiv\varepsilon L\pmod Q\) and
\(d_{yb}=\normQ{Ly}\).  Since \(b\ge2y\),
\(Ly\le(Q+1)/2\).  If \(Ly\le Q/2\), then
\(bd_{yb}=y(Q+\varepsilon)\ge y(Q-1)\).  If \(Ly>Q/2\), integrality
and the oddness of \(Q\) (because \(2\in A\) is a unit) force
\(Ly=(Q+1)/2\).  Hence
\[
 b=\frac{Q+\varepsilon}{L}
   =2y\frac{Q+\varepsilon}{Q+1}.
\]
If \(\varepsilon=-1\), this is strictly smaller than \(2y\), contrary to
\(b\in F\).  Therefore \(\varepsilon=1\) and \(b=2y\); hence
\(bd_{yb}=b(Q-1)/2=y(Q-1)\) again.  Thus
\begin{equation}\label{eq:far-baseline}
 \frac1{d_{yb}}\le\frac{b}{y(Q-1)}\qquad(b\in F).
\end{equation}
The total numerator weight available to \(F\) is at most
\(Q-y-p_0\).  Using this budget and the fact that
\(p/y+(Q-y-p)/(y(Q-1))\) increases with \(p\), we obtain
\begin{equation}\label{eq:P-F}
 \sum_{b\in P\cup F}\frac1{d_{yb}}
 \le\frac{p_0}{y}+\frac{Q-y-p_0}{y(Q-1)}\le\frac2m.
\end{equation}
For the final inequality, the function
\(\phi(p)=p/y+(Q-y-p)/[y(Q-1)]\) is increasing, and at
\(p=2y/m-1\) its gap from \(2/m\) is
\[
 \frac2m-\phi\!\left(\frac{2y}{m}-1\right)
 =\frac{y+(2y/m-1)-1}{y(Q-1)}>0.
\]

\smallskip
\noindent\emph{Smaller dyadic nondivisors.}
Every \(b\in X\) is \(b=2^{v+j}\), \(1\le j\le J\).  Choose
\(\varepsilon\in\{\pm1\}\) and \(L\) with
\(Q=L2^{v+j}+\varepsilon\), and write
\(m=\rho2^j+r\), \(0<r<2^j\).  From
\[
 D2^vm+\eta=L2^{v+j}+\varepsilon
\]
we get the integer
\[
 c=\frac{\eta-\varepsilon}{2^v},\qquad
 h=\frac{Dr+c}{2^j},\qquad 1\le h\le D-1.
\]
The preceding equality shows that \(\eta-\varepsilon\) is divisible by
\(2^v\), so \(c\) is an integer.  After division by \(2^v\), it reads
\[
 L2^j=Dm+c=D\rho2^j+Dr+c,
\]
so \(h=(Dr+c)/2^j=L-D\rho\) is an integer.  Since \(|c|\le2\),
\(D\ge3\), and \(1\le r<2^j\),
\[
 Dr+c\ge D-2\ge1,
 \qquad
 Dr+c\le D(2^j-1)+2<D2^j.
\]
Thus \(1\le h\le D-1\), and \(L=D\rho+h\).  Since
\(2^{v+j}L\equiv-\varepsilon\pmod Q\), one has
\((2^{v+j})^{-1}\equiv-\varepsilon L\pmod Q\), and the distance is
\(d_{yb}=\normQ{Ly}\).  Moreover,
\[
 Ly=(D\rho+h)y=\rho(Q-\eta)+hy
 \equiv hy-\rho\eta\pmod Q.
\]
The following two integers sum to \(Q\), and hence are the complementary
representatives of this residue:
\[
 hy-\rho\eta,\qquad (D-h)y+(\rho+1)\eta.
\]
The first is at least \(y-\rho\), while the second is at least
\(y-\rho-1\).  Thus both are positive and at least \(y-\rho-1\).  Since
\(2\rho\le m-1\le y-1\), we have
\(\rho\le(y-1)/2\), and therefore
\[
 d_{yb}\ge y-\rho-1\ge\frac{y-1}{2},
 \qquad \frac1{d_{yb}}\le\frac2{y-1}<\frac{10}{3y}.
\]
There are at most \(J\) such dyadic elements, giving the middle term of
\eqref{eq:dyadic-close}.

\smallskip
\noindent\emph{The close interval.}
Finally, let \(b=y+k\in\mathcal C\), and choose
\(\varepsilon\in\{\pm1\}\), \(t\ge3\), with
\(Q=t(y+k)+\varepsilon\).  The bound \(t\ge3\) follows from
\(Q>3b=3(y+k)\): if \(\varepsilon=1\), then
\(t=(Q-1)/b>3-1/b\), and if \(\varepsilon=-1\), the bound is even
larger.  Direct reduction gives
\[
 d_{yb}=tk+\varepsilon=(D-t)y+\eta.
\]
Indeed, subtracting \(Q=t(y+k)+\varepsilon\) from
\(Q=Dy+\eta\) gives the equality of the two displayed expressions.
Moreover, \(Q-d_{yb}=ty\), and
\(ty>d_{yb}\) is equivalent to \(t(y-k)>\varepsilon\), which holds
because \(1\le k<y\) and \(t\ge3\).  Hence the displayed positive
integer is the least cyclic representative.  Put \(\ell=D-t\).  The
identity \(d_{yb}=\ell y+\eta>0\) first gives \(\ell\ge0\).  Also
\(d_{yb}=tk+\varepsilon\ge3\cdot1-1=2\), whereas \(\ell=0\) would give
\(d_{yb}=\eta\le1\).  Hence \(\ell\ge1\).
For fixed \(\ell\) and \(\varepsilon\), the integer \(t=D-\ell\) is
fixed and \(Q=t(y+k)+\varepsilon\) determines at most one \(k\), hence at
most one \(b\).  Since \(d_{yb}=\ell y+\eta\ge\ell(y-1)\), summation over
the two signs gives
\begin{equation}\label{eq:C-first}
 \sum_{b\in\mathcal C}\frac1{d_{yb}}
 \le\frac{2\Harm_{D-1}}{y-1}.
\end{equation}
For a second estimate, direct algebra gives
\[
 d_{yb}>\frac{k(D-2)}2.
\]
After multiplying the desired inequality by the positive number
\(2(y+k)\), direct substitution gives
\[
 2(y+k)\left(d_{yb}-\frac{k(D-2)}2\right)
 =Dk(y-k)+2k(y+k)+2k\eta+2\varepsilon y.
\]
Since \(\eta,\varepsilon\ge-1\), the right side is at least
\[
 Dk(y-k)+2(k-1)(y+k)>0.
\]
Therefore
\begin{equation}\label{eq:C-second}
 \sum_{b\in\mathcal C}\frac1{d_{yb}}
 <\frac{2\Harm_{y-1}}{D-2}.
\end{equation}
If \(D\le y\), then \(\Harm_{D-1}\le\Harm_{y-1}\), so use
\eqref{eq:C-first}.  If \(D\ge y+1\), then \(D-2\ge y-1\), so use
\eqref{eq:C-second}.  In either case the close contribution is at most
\(2\Harm_{y-1}/(y-1)\).  Together with
\eqref{eq:P-F}, this proves \eqref{eq:dyadic-close}.
\end{proof}

A bad element cannot divide the good element \(y\), so its \(y\)-row
product is at least \(Q-y\).  Lemma~\ref{lem:s1-minimum-row} and
Proposition~\ref{prop:kernel} give
\begin{equation}\label{eq:y-full-envelope}
 M_y<H_y^{\rm good}+\frac{16n}{7(Q-y)}+\frac{8n}{7Q}
 <H_y^{\rm good}+\frac{32n}{7Q}.
\end{equation}

This turns the harmonic estimate into a denominator cutoff.  We use one
estimate when \(y\) is large and a sharper, \(Q\)-dependent estimate on the
remaining bounded range of \(y\).  The thresholds are chosen for these two
regimes: when \(y\ge97\), the good contribution falls below \(4/5\), while
for \(5\le y\le96\) and \(Q\ge5608\), it falls below \(32/35\).  In each
case the resulting linear bound on \(Q\), combined with the distinctness
bound, contradicts \(Q\ge5608\).

\begin{lemma}[Analytic cutoff in the critical branch]
\label{lem:critical-cutoff}
If \eqref{eq:critical-branch} holds, a selected nondyadic good element
exists, and every row passes, then \(Q\le5607\).
\end{lemma}

\begin{proof}
Assume \(Q\ge5608\).

\smallskip
\noindent\emph{The large range of \(y\).}
First suppose \(y\ge97\).  If \(m\ge5\), then
\(J/y\le6/97\): this is immediate when \(J\le6\), while for
\(J\ge7\) it follows from \(m\ge2^J+1\) and the decrease of
\(J/(2^J+1)\); its value at \(J=7\) is \(7/129<6/97\).
The stated decrease for \(J\ge2\) follows after cross multiplication from
\((J-1)2^J-1>0\).
Also \(\Harm_j/j\) decreases, since
\(\Harm_{j+1}/(j+1)<\Harm_j/j\) is equivalent to
\(j/(j+1)<\Harm_j\).  Grouping the terms with denominators in
\([2^k,2^{k+1}-1]\) gives \(\Harm_{96}<7<8\).
Lemma~\ref{lem:dyadic-close} yields
\[
 H_y^{\rm good}<\frac25+\frac{20}{97}+\frac16
 =\frac{2249}{2910}<\frac45.
\]
If \(m=3\), then \(y\ge192\), \(J=1\), and the same dyadic grouping
gives \(\Harm_{191}<8\), whence
\[
 H_y^{\rm good}<\frac23+\frac5{288}+\frac{16}{191}<\frac45.
\]
Equation~\eqref{eq:y-full-envelope} would force
\(Q<160n/7\).  Together with \eqref{eq:s1-distinct-sum}, this gives
\(n\le42\): at \(n=43\), the distinctness lower bound is \(990\),
while \(160n/7<983\), and thereafter the quadratic bound grows faster.
Substitution then gives \(Q<960\), hence \(Q\le959\), a contradiction.

\smallskip
\noindent\emph{The bounded range of \(y\).}
It remains that \(5\le y\le96\).  For \(b\in X\), the product
\(d_{yb}b\) is at least \(Q-y\), and the sum of all possible dyadic
\(b<y\) is less than \(2y\).  Thus
\[
 \sum_{b\in X}\frac1{d_{yb}}
 <\frac{2y}{Q-y}\le\frac{192}{Q-96}.
\]
Since
\(D-2=(Q-\eta-2y)/y\ge(Q-2y-1)/y\), the second close estimate in the
proof of Lemma~\ref{lem:dyadic-close} gives
\[
 \sum_{b\in\mathcal C}\frac1{d_{yb}}
 <\frac{2y\Harm_{y-1}}{Q-2y-1}
 \le\frac{1152}{Q-193}.
\]
To justify the last numerical bound, write
\[
 \Harm_{95}
 =\Harm_3+
  \sum_{k=2}^{5}\sum_{r=2^k}^{2^{k+1}-1}\frac1r
  +\sum_{r=64}^{95}\frac1r.
\]
In each full dyadic block, split the \(2^k\) terms into equal halves.  The
first half contributes at most \(2^{k-1}/2^k=1/2\), while every denominator
in the second half is at least \(3\cdot2^{k-1}\), so that half contributes
at most \(1/3\).  The final \(32\) terms contribute at most \(32/64=1/2\).
Consequently
\[
 \Harm_{95}
 \le\frac{11}{6}+4\left(\frac12+\frac13\right)+\frac12
 =\frac{17}{3}<6.
\]
By \eqref{eq:P-F},
\(\sum_{b\in P\cup F}1/d_{yb}\le2/m\le2/3\); evaluation at
\(Q=5608\) gives
\[
 H_y^{\rm good}
 <\frac23+\frac{192}{5512}+\frac{1152}{5415}
 =\frac{3410978}{3730935}<\frac{32}{35}.
\]
The final comparison is exact; its positive gap is
\[
 \frac{32}{35}-
 \left(\frac23+\frac{192}{5512}+\frac{1152}{5415}\right)
 =\frac{1138}{26116545}>0.
\]
Now \eqref{eq:y-full-envelope} and a passing row imply
\(Q<160n/3\).  Combining this with \eqref{eq:s1-distinct-sum} first
gives \(n\le103\): at \(n=104\), the quadratic lower bound is \(5565\),
whereas \(160n/3<5547\), and again the lower bound subsequently grows
faster.  It follows that
\(Q<16480/3<5494\), again a contradiction.
\end{proof}

\subsection{Exact rational refinement}

The exact rational envelope below sharpens the analytic cutoff from
\(Q\le5607\) to \(Q\le289\), placing the critical branch inside the global
finite search.
For an odd integer \(Q\ge3\), define
\[
 \mathcal G_Q=\{b\in\Z:2\le b<Q/3,\ b\mid Q-1\text{ or }b\mid Q+1\}.
\]
Every \(b\in\mathcal G_Q\) is coprime to \(Q\).  In the formulas below,
\(b^{-1}\) denotes the unique representative in \(\{1,\ldots,Q-1\}\)
satisfying \(bb^{-1}\equiv1\pmod Q\).
Let \(n_Q\) be the largest integer satisfying
\((n_Q+1)(n_Q+2)/2\le Q\).  For nondyadic \(y\in\mathcal G_Q\), write
\(2^v\Vert y\), put \(p=2^{v+1}-1\), and set
\[
 \mathcal E_{Q,y}=\{b\in\mathcal G_Q\setminus\{y\}:
 [b<y,\ b\text{ dyadic},\ b\nmid y]\text{ or }y<b<2y\}.
\]
Define
\begin{align}
 U_{\rm g}(Q,y)
 &=\frac py+\frac{Q-y-p}{y(Q-1)}
 +\sum_{b\in\mathcal E_{Q,y}}
 \max\left\{0,\frac1{\normQ{yb^{-1}}}
 -\frac{b}{y(Q-1)}\right\},\label{eq:Ug}\\
 U(Q,y)&=U_{\rm g}(Q,y)
 +\frac{16n_Q}{7(Q-y)}+\frac{8n_Q}{7Q}.
 \label{eq:U-total}
\end{align}

The terms in these envelopes have distinct roles.  The first two terms of
\(U_{\rm g}(Q,y)\) combine the maximal possible dyadic-divisor contribution
with the far-shell baseline from
\eqref{eq:far-baseline}.  The finite sum then adds every positive promotion
above that baseline that can arise from an exceptional candidate in
\(\mathcal E_{Q,y}\).  Thus \(U_{\rm g}(Q,y)\) is a relaxed upper bound for
the reciprocal contribution of all good elements.  The two additional terms
in \(U(Q,y)\) bound, respectively, the contribution of the bad elements and
the error incurred when \(1/d\) is replaced by the inverse-sine kernel
\(K_Q(d)\).  The next lemma checks that every relaxation is in the
upper-bound direction and retains the strict inequality needed below.

\begin{lemma}[Validity of the refined envelope]\label{lem:U-valid}
If the row at \(1\) passes and \(y\) is the least selected
nondyadic good element, then
\[
 M_y<U(Q,y).
\]
\end{lemma}

\begin{proof}
Let \(p_0\le p\) be the sum of the selected dyadic divisors of \(y\).
Give every other selected good element the baseline
\(b/[y(Q-1)]\), and add its positive excess over that baseline whenever
it lies in the exceptional set \(\mathcal E_{Q,y}\).  By minimality of
\(y\), every remaining selected good element, other than the dyadic
divisors and the \(X\)/close exceptional elements, lies in \(F\); hence
\eqref{eq:far-baseline} applies to it.
Writing \(\mathcal E^{\rm sel}_{Q,y}\) for the selected exceptional
elements, the preceding assignments give the explicit inequality
\[
 H_y^{\rm good}\le
 \frac{p_0}{y}+\frac{Q-y-p_0}{y(Q-1)}
 +\sum_{b\in\mathcal E^{\rm sel}_{Q,y}}
 \max\left\{0,\frac1{d_{yb}}-\frac{b}{y(Q-1)}\right\}.
\]
Here the total remaining numerator budget is at most \(Q-y-p_0\).  The
baseline term may therefore use that entire budget, even if some of it is
actually occupied by bad elements or simply left unused; this only enlarges the
right side.  Moreover,
\[
 \frac{p_0}{y}+\frac{Q-y-p_0}{y(Q-1)}
\]
increases with \(p_0\), because the coefficient of \(p_0\) is
\((Q-2)/[y(Q-1)]>0\).  Replacing \(p_0\) by its upper bound \(p\) can
therefore only increase the estimate.  Likewise, adjoining every positive
promotion from \(\mathcal E_{Q,y}\), including promotions for unselected
candidates, is a relaxation in the upper-bound direction.  These two
relaxations prove \(H_y^{\rm good}\le U_{\rm g}(Q,y)\).
The last two terms in \eqref{eq:U-total} are respectively the bad-weight
and full sine-kernel bounds.  The former is strict because
\(W<16n/7\), and the latter is strict by
Proposition~\ref{prop:kernel}.  They therefore give the strict bound with
\(n\) in place of \(n_Q\).  Finally, \(n\le n_Q\) follows from the
distinctness inequality \((n+1)(n+2)/2\le Q\).  Thus replacing \(n\) by
\(n_Q\) again enlarges the upper bound, and \(M_y<U(Q,y)\) follows.
\end{proof}

Thus the critical branch can survive only when the exact rational envelope
is at least one.  The corresponding finite calculation leaves no denominator
above \(289\).

\begin{proposition}[Critical finite refinement]\label{prop:critical-289}
Under the hypotheses of Theorem~\ref{thm:SC}, suppose that
\(\min A=1\), \(2\in A\), \(3\notin A\), and that \(A\) contains a
nondyadic good element.  If every row passes, then \(Q\le289\).
\end{proposition}

\begin{proof}
Lemma~\ref{lem:critical-cutoff} gives \(Q\le5607\).  Let \(y\) be the
least selected nondyadic good element.  Since \(3\notin A\), while \(2\)
and \(4\) are dyadic, we have \(y\ge5\).  The subcriticality hypothesis
gives \(3y<Q\), and goodness gives \(y\mid Q-1\) or \(y\mid Q+1\).
Moreover, \(Q\) is odd because \(2\in A\subset\U\); in particular,
\(Q\ge17\).  Thus the actual pair \((Q,y)\) occurs among the pairs
enumerated in Appendix~\ref{app:critical-verifier}.  That appendix
evaluates the rational envelope \(U(Q,y)\) exactly and proves that every
eligible pair with \(Q\ge290\) satisfies \(U(Q,y)<1\).  On the other hand,
a passing row has \(M_y\ge1\), whereas Lemma~\ref{lem:U-valid} gives
\(M_y<U(Q,y)\).  Hence \(Q\ge290\) is impossible.
\end{proof}

\subsection{Global finite reduction and completion of subcritical rigidity}

All four \(s=1\) branches lie in \(4\le Q\le715\).  For each
denominator, membership in \(A\) becomes a binary selection variable.  The
numerator budget and the all-passing row inequalities form a finite integer
system, and an exact branch-and-knapsack search proves that no such system is
feasible.  This completes the exceptional case and, together with \(s>1\),
the subcritical rigidity theorem.

\begin{theorem}[The \(s=1\) case]\label{thm:s-one}
Under the hypotheses of Theorem~\ref{thm:SC}, if \(\min A=1\), then some
row is deficient.
\end{theorem}

\begin{proof}
\noindent\emph{An exact necessary system.}
We first formulate the finite system that any counterexample would satisfy.
For each integer
\(Q\) with \(4\le Q\le715\), let
\[
 \mathcal P_Q=\{a\in\Z:1\le a<Q/3,\ (a,Q)=1\}.
\]
For every \(b\in\mathcal P_Q\), let
\(b^{-1}\in\{1,\ldots,Q-1\}\) be the unique representative satisfying
\(bb^{-1}\equiv1\pmod Q\).  For distinct
\(a,b\in\mathcal P_Q\), define
\[
 d_{ab}=\normQ{ab^{-1}},\qquad
 v_{ab}=\left\lceil \Lambda\left(\frac1{d_{ab}}
 +\frac{16d_{ab}}{7Q^2}\right)\right\rceil,
 \qquad \Lambda=10^7.
\]
The necessary integer system has binary variables
\(z_a\in\{0,1\}\) for \(a\in\mathcal P_Q\), with
\begin{equation}\label{eq:finite-system}
 z_1=1,\qquad
 \sum_{a\in\mathcal P_Q}az_a\le Q,
 \qquad
 \sum_{b\in\mathcal P_Q\setminus\{a\}}v_{ab}z_b
 \ge \Lambda z_a\quad(a\in\mathcal P_Q).
\end{equation}
Here \(z_a=1\) means \(a\in A\).  The anchor condition is \(z_1=1\),
and the middle inequality is precisely the relaxed numerator budget.  For a
selected \(a\), passing gives
\[
 1\le M_a=\sum_{b\ne a}K_Q(d_{ab}),
 \qquad d_{ab}=\normQ{ab^{-1}}.
\]
Proposition~\ref{prop:kernel} is strict, so
\[
 1<\sum_{b\ne a}
 \left(\frac1{d_{ab}}+\frac{16d_{ab}}{7Q^2}\right)z_b.
\]
Since \(v_{ab}\) is obtained by multiplying each parenthesized quantity by
\(\Lambda\) and rounding upward,
\[
 \sum_{b\ne a}v_{ab}z_b>\Lambda.
\]
Therefore the sum is certainly at least the integer \(\Lambda\).  If
\(z_a=0\), the final
inequality of \eqref{eq:finite-system} is automatic.  Consequently,
every genuine counterexample yields a feasible solution of
\eqref{eq:finite-system}.

Appendix~\ref{app:global-verifier} gives the exact branch-and-knapsack
exhaustion of \eqref{eq:finite-system}.  It proves that the system is
infeasible for every \(Q\) with \(4\le Q\le715\).

\smallskip
\noindent\emph{Exhaustion of the analytic branches.}
It remains to check that the analytic branches lie in this interval.
Lemma~\ref{lem:s1-minimum-row} guarantees that the least good element \(x\)
exists in any all-passing set.  The following four-way case split is
disjoint and exhaustive:
\begin{enumerate}[label=\textup{(\alph*)},leftmargin=2.2em]
\item \(x\ge3\), when Lemma~\ref{lem:s1-x-ge-three} gives \(Q\le703\);
\item \(x=2\) and \(3\in A\), when Lemma~\ref{lem:s1-three} gives
\(Q\le715\);
\item \(x=2\), \(3\notin A\), and every selected good element is dyadic,
when Lemma~\ref{lem:no-nondyadic} gives \(Q\le171\);
\item \(x=2\), \(3\notin A\), and a nondyadic good element is selected,
when Proposition~\ref{prop:critical-289} gives \(Q\le289\).
\end{enumerate}
The first decision is whether \(x=2\); in that case the second is whether
\(3\) is selected, and in the remaining case the third is whether a
nondyadic good element exists.  Hence no case is omitted.  Every branch is
therefore eliminated by
\eqref{eq:finite-system}, a contradiction.
\end{proof}

\begin{proof}[Proof of Theorem~\ref{thm:SC}]
Let \(s=\min A\).  Since \(A\) is nonempty, either \(s>1\) or \(s=1\).
Theorem~\ref{thm:s-positive} handles the first case, and
Theorem~\ref{thm:s-one} handles the second.  These cases are exhaustive, so
every subcritical pair has a deficient row, as asserted.
\end{proof}

\section{From the one-third theorem to the binary classification}
\label{sec:induction}

Theorem~\ref{thm:SC} and Corollary~\ref{cor:one-third} now establish the
one-third theorem, so the analytic part of the proof is complete.  We turn
that single large density into the full binary classification without further
finite verification.  Deletion preserves balance, periodicity puts the
compressed system back in the rational Beatty category, and induction makes
the survivors binary up to a common scale.  A rigid two-set extension then
determines both that scale and the deleted numerator.

\subsection{Periodic balanced sets}

The compressed indicators are periodic and balanced.  The following
classical mechanical-word characterization identifies them as
integer-translated rational Beatty sets while retaining the period and shift
\cite[pp.~470--471]{TijdemanBalanced2000}.

\begin{lemma}[Periodic balanced sets are rational Beatty sets]
\label{lem:periodic-balanced}
Every nonempty proper periodic balanced subset of \(\Z\) is an
integer-translated rational Beatty set.
\end{lemma}

\begin{proof}
Let \((w_t)_{t\in\Z}\) be its binary word.  Choose a period \(Q\), and
let \(P\) be the number of ones in one period.  Nonemptiness and
properness give \(0<P<Q\).  Put \(\delta=P/Q\).  For \(L\ge1\), let
\[
 A_t(L)=\sum_{j=0}^{L-1}w_{t+j}
\]
be the number of ones in the length-\(L\) factor beginning at \(t\).
Periodicity gives
\[
 \sum_{t=0}^{Q-1}A_t(L)
 =\sum_{j=0}^{L-1}\sum_{t=0}^{Q-1}w_{t+j}=LP.
\]
Thus the average of these \(Q\) integer counts is \(L\delta\).  Balance
says that any two differ by at most one, and therefore
\begin{equation}\label{eq:factor-count-floor-ceil}
 A_t(L)\in\{\lfloor L\delta\rfloor,\lceil L\delta\rceil\}.
\end{equation}
Indeed, if \(a\) is the least of the integer counts, then all of them lie
in \(\{a,a+1\}\).  Their average lies in \([a,a+1]\), which forces
\(a=\lfloor L\delta\rfloor\) unless the average is an integer; in the
integer case all counts equal that average.

\smallskip
\noindent\emph{A mechanical representation.}
Define \(F(0)=0\),
\[
 F(n)=\sum_{0\le t<n}w_t\quad(n>0),\qquad
 F(n)=-\sum_{n\le t<0}w_t\quad(n<0).
\]
For \(n>m\), this convention ensures that
\(F(n)-F(m)=\sum_{t=m}^{n-1}w_t\).  Every interval of \(Q\) consecutive
positions contains \(P\) ones, so \(F(n+Q)-F(n)=P\).  Consequently
\[
 D_n:=F(n)-n\delta
 \quad\text{satisfies}\quad D_{n+Q}=D_n.
\]
If \(n>m\) and \(L=n-m\), then \eqref{eq:factor-count-floor-ceil}
shows that
\[
 D_n-D_m=F(n)-F(m)-L\delta\in(-1,1).
\]
By symmetry, \(|D_n-D_m|<1\) for all \(m,n\).

The periodic sequence \(D_n\) takes only finitely many values.  Choose
\(N\) with \(C:=D_N=\max_nD_n\).  Then \(0\le C-D_n<1\).  Since
\(n\delta+C=F(n)+(C-D_n)\) and \(F(n)\) is an integer, we obtain
\begin{equation}\label{eq:mechanical-floor}
 F(n)=\lfloor n\delta+C\rfloor,\qquad
 w_n=F(n+1)-F(n).
\end{equation}

\smallskip
\noindent\emph{From jumps to a Beatty set.}
Because \(C=D_N=F(N)-N\delta\), the second identity becomes
\begin{equation}\label{eq:shifted-mechanical-word}
 w_n=\lfloor(n+1-N)\delta\rfloor-
     \lfloor(n-N)\delta\rfloor.
\end{equation}
Write \(\delta=P_0/Q_0\) in lowest terms and put
\(\alpha=Q_0/P_0=1/\delta>1\).  If \(k=n-N\), the right side of
\eqref{eq:shifted-mechanical-word} counts the integers \(\ell\) in
\(k\delta<\ell\le(k+1)\delta\).  Since \(0<\delta<1\), there is at most
one.  Such an integer exists exactly when
\(k<\ell\alpha\le k+1\), or \(k=\lceil\ell\alpha\rceil-1\).  Hence the
jump positions are
\begin{equation}\label{eq:jump-ceil-set}
 N-1+\{\lceil\ell\alpha\rceil:\ell\in\Z\}.
\end{equation}

It remains to replace ceilings by floors using an integer translation.
If \(P_0=1\), take \(j=0\) and \(t=1\).  Then
\(jQ_0+1=tP_0\), and the required floor--ceiling identity is immediate
because \(\alpha=Q_0\) is integral.  If \(P_0>1\), multiplication by
\(Q_0\) permutes the residue classes modulo \(P_0\), so choose \(j\) with
\(jQ_0\equiv-1\pmod{P_0}\), and write \(jQ_0+1=tP_0\).  For every
\(\ell\in\Z\), direct separation of the case \(P_0\mid\ell\) gives
\begin{equation}\label{eq:ceil-floor-identity}
 \lfloor(\ell+j)\alpha\rfloor
 =\lceil\ell\alpha\rceil+t-1.
\end{equation}
For a direct check, write \(\ell Q_0=uP_0+r\), where
\(0\le r<P_0\).  Since \(j\alpha=t-1/P_0\), if \(r=0\) both sides of
\[
 \left\lfloor\ell\alpha-\frac1{P_0}\right\rfloor
 =\lceil\ell\alpha\rceil-1
\]
equal \(u-1\), while if \(r>0\) both equal \(u\).  This proves
\eqref{eq:ceil-floor-identity}.  As \(\ell\) ranges over \(\Z\), so does
\(\ell+j\).  Substitution in \eqref{eq:jump-ceil-set} identifies the set
of ones with
\[
 N-t+\{\lfloor\ell Q_0/P_0\rfloor:\ell\in\Z\}
 =\B^{Q_0}_{P_0,N-t},
\]
as required.  The shift \(N-t\) is an integer.
\end{proof}

Thus a nonempty proper compressed indicator returns to the rational Beatty
category as soon as periodicity and balance are known.  In the final
subsection we verify these conditions for every survivor and then apply
induction.  We first isolate the two-set criterion that makes the subsequent
reinsertion rigid.

\subsection{A rigid two-set extension}

After deletion and induction, the surviving numerators are
\(h,2h,\ldots,Mh\); only the deleted numerator remains unknown.  It is
enough to compare it with the survivor \(Mh\).  Their actual disjointness
feeds a two-sequence Diophantine criterion, and the power-of-two scale makes
that condition rigid.

The following theorem characterizes when a pair admits disjoint
shifts.  It is Morikawa's Japanese remainder theorem
\cite[Theorem~1]{Morikawa1985}, with Simpson's simpler proof
\cite[Theorem~8]{Simpson2004}; here \((x,y)\) denotes \(\gcd(x,y)\).
Lemma~\ref{lem:rigid-extension} uses only this pairwise statement, so no
simultaneous choice of shifts for the full family is required.
In the notation \(\B^q_{p,*}\), the asterisk means that the integer shift may
be chosen independently for each set.

\begin{theorem}[Morikawa--Simpson criterion]\label{thm:morikawa}
For \(0<p,s<q\), consider the two sets \(\B^q_{p,*}\) and
\(\B^q_{s,*}\), and set
\[
 g_p=(p,q),\quad A=q/g_p,\quad x=p/g_p,
 \qquad g_s=(s,q),\quad B=q/g_s,\quad y=s/g_s,
\]
and put \(d=(x,y)\), \(u=x/d\), \(v=y/d\).  Some translates of the two
rational Beatty sets are disjoint if and only if there exist positive
integers \(k,\ell\) such that
\begin{equation}\label{eq:morikawa}
 ku+\ell v=(A,B)-2uv(d-1).
\end{equation}
\end{theorem}

To match the notation with the cited theorem, observe that after reduction
\[
 \frac qp=\frac{q/g_p}{p/g_p}=\frac Ax,
 \qquad
 \frac qs=\frac{q/g_s}{s/g_s}=\frac By,
\]
with \((A,x)=(B,y)=1\).  Simpson's reduced numerators are \(A,B\), his
reduced denominators are \(x,y\), and their gcds are \((A,B)\) and
\(d=(x,y)\).  His coprime residual denominators are exactly
\(u=x/d\) and \(v=y/d\), so his criterion is precisely
\eqref{eq:morikawa}.
Simpson states the criterion for real shifts.  The rational-shift
normalization in Section~\ref{sec:normalization} shows that existence for
real shifts is equivalent here to existence for integer translates, so the
displayed formulation loses no cases.

We now specialize the criterion to a binary family with one missing scale.
The inequality \(2p\ge Dh\) below is exactly the one-third bound after
deletion; coprimality records the primitive normalization.

\begin{lemma}[Rigid binary extension]\label{lem:rigid-extension}
Let \(M\ge2\) be a power of two and \(D=2M-1\).  Suppose that positive
integers \(p,h,q\) satisfy
\begin{equation}\label{eq:extension-data}
 q=p+Dh,\qquad (p,h)=1,\qquad 2p\ge Dh,\qquad p\ne Mh.
\end{equation}
If the two families \(\B^q_{p,*}\) and \(\B^q_{Mh,*}\) admit disjoint
translations,
then
\[
 (p,h,q)=(2M,1,4M-1).
\]
\end{lemma}

\begin{proof}
The hypotheses imply \(0<p<q\), and also
\(q-Mh=p+(M-1)h>0\); hence both numerators lie in the range required by
Theorem~\ref{thm:morikawa}.

\smallskip
\noindent\emph{Arithmetic reduction.}
Set
\[
 c=(p,q)=(p,D),\qquad t=(Mh,q)=(M,q).
\]
Every gcd equality follows directly from \(q=p+Dh\):
\[
 (p,q)=(p,p+Dh)=(p,Dh)=(p,D),
\]
where the final equality uses \((p,h)=1\).  Likewise
\((h,q)=(h,p+Dh)=(h,p)=1\), so \((Mh,q)=(M,q)\).
Thus \(c\) is an odd divisor of \(D\), \(t\) is a power of two dividing
\(M\), and \((c,t)=1\).  Since \(c,t\mid q\) are coprime,
\(\operatorname{lcm}(c,t)=ct\), and hence
\((q/c,q/t)=q/(ct)\).  In Theorem~\ref{thm:morikawa}, the complete
dictionary is
\[
 g_p=c,\quad g_s=t,\quad A=q/c,\quad B=q/t,
 \quad x=p/c,\quad y=Mh/t,
 \quad d=(x,y),\quad u=x/d,\quad v=y/d.
\]
Thus disjointness requires positive \(k,\ell\)
with
\begin{equation}\label{eq:extension-morikawa}
 ku+\ell v=\frac{q}{ct}-2uv(d-1).
\end{equation}

\smallskip
\noindent\emph{The common dyadic factor.}
First we prove \(t=1\).  If \(t>1\), then the power of two \(t\mid q\)
is even, so \(q\) is even.  Since \(q=p+Dh\) and \(D\) is odd,
\(p,h\) have the same parity.  Coprimality excludes both being even, so
both are odd.  Therefore \(x=p/c\) is odd, whereas
\(y=(M/t)h\) is a power of two times \(h\).  Because
\((x,h)=1\) and \((x,M/t)=1\), we get \(d=(x,y)=1\).  Multiplying
\eqref{eq:extension-morikawa} by \(ct\) gives
\begin{equation}\label{eq:extension-even}
 kpt+\ell Mhc=p+Dh.
\end{equation}
Since \(k,\ell\ge1\), the left side of \eqref{eq:extension-even} is at
least \(pt+Mhc\).  If \(c>1\), then \(c\ge3\) and
\[
 pt+Mhc-(p+Dh)
 =(t-1)p+\bigl(M(c-2)+1\bigr)h>0,
\]
contradicting \eqref{eq:extension-even}.  If \(c=1\), the same lower bound
in \eqref{eq:extension-even} gives
\((t-1)p\le(M-1)h\), whereas \(t\ge2\) and
\(p\ge(M-\tfrac12)h\).  This is again impossible.

\smallskip
\noindent\emph{The residual gcd.}
Now
\[
 d=(p/c,Mh)=(p/c,M),
\]
because \((p/c,h)=1\).  Thus \(d\mid M\).  Write
\[
 e=M/d,\qquad v=eh,\qquad p=cdu.
\]
Since \(k,\ell\ge1\), equation \eqref{eq:extension-morikawa} implies
\[
 \frac qc-2uv(d-1)\ge u+v.
\]
Substitute \(q/c=du+(D/c)h\) and \(v=eh\):
\begin{align*}
 du+\frac Dc h-2u(eh)(d-1)&\ge u+eh,\\
 h\left(\frac Dc-e\right)&\ge(d-1)u(2eh-1).
\end{align*}
This is the necessary inequality
\begin{equation}\label{eq:extension-necessary}
 (d-1)u(2v-1)\le h\left(\frac Dc-e\right).
\end{equation}

We exclude the possibilities \(d=1\) and \(1<d<M\) in turn; since \(d\)
divides the power of two \(M\), the only remaining value will be \(d=M\).

If \(d=1\), the left side of \eqref{eq:extension-necessary} is zero, so
\(D/c-M\ge0\).  If \(c>1\), then the odd integer \(c\ge3\) and
\(D/c\le(2M-1)/3<M\), a contradiction.  Hence \(c=1\).  Since now
\(t=c=d=1\) and \(p=u\), equation \eqref{eq:extension-morikawa} is
directly
\(kp+\ell Mh=p+Dh\).  Subtracting
\(p+Mh\) from both sides gives
\[
 (k-1)p+M(\ell-1)h=(M-1)h.
\]
Because \(p\ge(M-\tfrac12)h>(M-1)h\), this forces \(k=1\), and then
\(M(\ell-1)=M-1\), impossible.  Thus \(d>1\).

If \(1<d<M\), then \(d,e\ge2\).  From
\eqref{eq:extension-data},
\(2cdu=2p\ge Dh\), so \(u\ge Dh/(2cd)\).  Substituting this lower
bound into the left side of \eqref{eq:extension-necessary}, and then
multiplying by the positive number \(2cd/h\), gives
\[
 D(d-1)(2eh-1)\le2d(D-ce).
\]
Since \(h\ge1\), we have \(2eh-1\ge2e-1\).  Replacing the left factor
by this smaller one preserves the inequality.  Multiplying by \(d\) and
using \(M=ed\) yields
\begin{equation}\label{eq:extension-polynomial}
 D(d-1)(2M-d)\le2d(Dd-Mc)\le2d(Dd-M).
\end{equation}
We now show that the reverse strict inequality always holds.  Write the
leftmost and rightmost expressions as \(L\) and \(R\).  Substituting
\(M=ed\) and \(D=2ed-1\), direct expansion gives
\[
 L-R=d\bigl((4e^2-6e)d^2+(-4e^2+2e+3)d+2e-1\bigr).
\]
With \(E=e-2\ge0\) and \(Y=d-2\ge0\), the expression in parentheses is
\[
 (4E^2+10E+4)Y^2+(12E^2+26E+7)Y+(8E^2+14E+1),
\]
whose coefficients are positive and whose constant term is \(1\).
Therefore \(L-R>0\), contradicting \eqref{eq:extension-polynomial}.  The
only divisor of the power of two \(M\) remaining after \(d>1\) is
\(d=M\).

\smallskip
\noindent\emph{The final scale.}
Now \(e=1\) and \(v=h\), so \eqref{eq:extension-necessary} reads
\[
 (M-1)u(2h-1)\le h\left(\frac Dc-1\right).
\]
If \(c>1\), then \(c\ge3\).  The left side is at least \(h(M-1)\),
whereas the right side is at most \(h(2M-4)/3\).  Their coefficient
difference is
\[
 (M-1)-\frac{2M-4}{3}=\frac{M+1}{3}>0,
\]
so this is impossible.  Thus \(c=1\), and division by \(M-1>0\) gives
\(u(2h-1)\le2h\).  If \(h\ge2\), then \(u=1\), but
\(p=cdu=M\), while
\(Dh=(2M-1)h\ge4M-2>2M=2p\), contrary to
\eqref{eq:extension-data}.  Therefore
\(h=1\), \(u\le2\), and \(p=Mu\ne M\) gives \(u=2\).  Consequently
\(p=2M\) and \(q=p+D=4M-1\).
\end{proof}

\subsection{Induction from the one-third theorem}

The reconstruction starts by choosing a largest numerator and deleting
its letter.  The compressed period and densities can be read directly from
one original period.  Induction determines the survivors as a binary family
with common scale \(h\); primitivity and the one-third bound supply the two
arithmetic constraints that let the rigid extension determine both \(h\)
and the deleted numerator.

\begin{proposition}\label{prop:one-third-induction}
Theorem~\ref{thm:one-third-main}, together with the known cases
\(3\le m\le7\), implies Theorem~\ref{thm:main} for every \(m\ge3\).
\end{proposition}

\begin{proof}
The cases \(3\le m\le6\) are covered by
Tijdeman~\cite{Tijdeman2000}, and the case \(m=7\) is due to Bar\'at and
Varj\'u~\cite[Theorem~1.6]{BaratVarju2003}.  To pass from the bi-infinite
formulation to these cited one-sided formulations, restrict each strictly
increasing Beatty sequence to its positive members and reindex that tail.
This preserves its modulus and gives a partition of \(\Z_{>0}\).
By the rationality argument and the selected
common-denominator normalization in Section~\ref{sec:normalization}, the
given bi-infinite partition admits primitive data
\((q;(p_i,r_i)_{i=1}^m)\) satisfying \eqref{eq:primitive}; thus every set
is \(\B^q_{p_i,r_i}\).
Assume the result for \(m-1\), with \(m\ge8\), and choose a largest
numerator \(p\).  By Theorem~\ref{thm:one-third-main}, its rate
\(q/p\le3\).  Lemma~\ref{lem:deletion} shows that erasing this letter
leaves a balanced word.

\smallskip
\noindent\emph{Deletion and compression.}
Enumerate the retained positions in increasing order as
\[
 \cdots<z_{-1}<z_0<z_1<\cdots.
\]
The original coding word is \(q\)-periodic, and every interval of \(q\)
consecutive positions contains exactly \(p\) copies of the deleted letter.
Hence translation by \(q\) preserves the retained set and advances its
order index by exactly \(q-p\):
\[
 z_{j+q-p}=z_j+q\qquad(j\in\Z).
\]
If the compressed word is \(\widetilde W_j=W_{z_j}\), then
\[
 \widetilde W_{j+q-p}=W_{z_j+q}=W_{z_j}=\widetilde W_j.
\]
Thus it has period \(q-p\).  The indicator of each remaining
letter is a nonempty proper periodic balanced binary word; by
Lemma~\ref{lem:periodic-balanced}, the compressed word is therefore a partition
into \(m-1\) rational Beatty sets.  More precisely, the compressed indices
\(j,j+1,\ldots,j+q-p-1\) correspond to the retained positions in the
half-open interval \([z_j,z_j+q)\).  That interval is one full original
period and contains exactly \(p_i\) copies of the \(i\)-th surviving
letter.  Its compressed density is therefore
\(p_i/(q-p)\).  Since the \(p_i\) are distinct, the surviving densities,
and hence their reciprocal moduli, remain distinct.

This change of modulus is Fraenkel's rescaling
\cite[Lemma~6]{Fraenkel1973}.  Lemma~\ref{lem:deletion} preserves balance
under deletion, and Lemma~\ref{lem:periodic-balanced} identifies the
compressed system as a rational Beatty partition.

\smallskip
\noindent\emph{Induction and reinsertion.}
By the induction hypothesis, let
\(M=2^{m-2}\) and \(D=1+2+4+\cdots+M=2M-1\).  The remaining moduli are
\[
 \left\{\frac{D}{2^j}:0\le j\le m-2\right\},
\]
so, on taking reciprocals, their densities are
\(\{2^j/D:0\le j\le m-2\}\).  For the density \(1/D\), some original
numerator \(p_i\) satisfies \(p_i/(q-p)=1/D\).  Hence the positive integer
scale is
\[
 h=\frac{q-p}{D}=p_i\in\Z_{>0},
\]
and the remaining original numerators are
\[
 h,2h,\ldots,Mh,\qquad q=p+Dh.
\]
Primitive normalization gives
\[
 1=\gcd(q,p,h,2h,\ldots,Mh)
 =\gcd(q,p,h)=\gcd(p+Dh,p,h)=\gcd(p,h).
\]
The one-third inequality is \(3p\ge q\), equivalently \(2p\ge Dh\),
and distinctness gives \(p\ne Mh\) (indeed, maximality gives
\(p>Mh\)).  The original sets with numerators \(p\) and \(Mh\) are
disjoint.  Lemma~\ref{lem:rigid-extension} therefore yields
\[
 p=2M,\qquad h=1,\qquad q=4M-1=2^m-1.
\]
Thus the original numerator set is
\(\{1,2,4,\ldots,M,2M\}=\{2^j:0\le j<m\}\), with
\(q=4M-1=2^m-1\), completing the induction.
\end{proof}

\begin{proof}[Proof of Theorem~\ref{thm:main}]
Theorem~\ref{thm:SC} and Corollary~\ref{cor:one-third} establish
Theorem~\ref{thm:one-third-main}.  Proposition~\ref{prop:one-third-induction}
then gives the binary classification for every \(m\ge3\).
\end{proof}

\subsection{The balanced-sequence corollary}
\label{subsec:balanced-corollary}

To prove Corollary~\ref{cor:balanced}, we first pass from a one-sided
balanced word to a Beatty partition.  The key is to treat each letter
indicator separately: after deleting a finite prefix, it agrees with the
indicator of a Beatty set.  The following lemma makes this passage precise.

\begin{lemma}[Eventual Beatty representation]\label{lem:eventual-beatty}
Let \(w=(w_n)_{n\ge0}\) be a balanced word over \(\{0,1\}\), and
suppose that the density of the letter \(1\) is \(\rho\in(0,1)\).
Then there exist \(\beta\in\mathbb R\) and an integer \(N\ge0\) such
that
\[
 w_n=1\quad\Longleftrightarrow\quad n\in S(1/\rho,\beta)
 \qquad(n\ge N).
\]
Moreover, if \(\rho=p/q\) in lowest terms, then \(w\) is eventually
\(q\)-periodic: \(w_{n+q}=w_n\) for all sufficiently large \(n\).
\end{lemma}

\begin{proof}

For \(n\ge0\), put \(F(n)=\sum_{0\le j<n}w_j\).  For \(L\ge1\)
and \(t\ge0\), let \(C_L(t)=F(t+L)-F(t)\), the number of ones in
the length-\(L\) factor starting at \(t\).

Fix \(L\), and let \(a_L\) and \(b_L\) be the minimum and maximum of
these counts.  Balance gives \(b_L-a_L\le1\).  The average count in the
\(h\) disjoint factors starting at \(0,L,\ldots,(h-1)L\) is
\(F(hL)/h\), which tends to \(L\rho\).  Hence
\(a_L\le L\rho\le b_L\), and therefore
\[
 \lvert C_L(t)-L\rho\rvert\le1
 \qquad(L\ge1,\ t\ge0).
\]
This estimate also controls the prefix discrepancies
\(D_n=F(n)-n\rho\).  Indeed, for \(n>t\) we have
\(D_n-D_t=C_{n-t}(t)-(n-t)\rho\), so
\[
 \lvert D_n-D_t\rvert\le1\qquad(n,t\ge0).
\]

\smallskip

Suppose first that \(\rho\) is irrational.  Set
\(\theta=\sup_{n\ge0}D_n\), which is finite by the preceding bound.
Then \(0\le\theta-D_n\le1\).  The equality \(\theta-D_n=1\) can
hold for at most one index, since two such indices would give
\((n-t)\rho=F(n)-F(t)\in\Z\).

Beyond this possible exceptional index, we have
\(0\le\theta-D_n<1\).  Since
\(n\rho+\theta=F(n)+(\theta-D_n)\), it follows that
\(F(n)=\lfloor n\rho+\theta\rfloor\).  Taking successive differences
gives, for all sufficiently large \(n\),
\[
 w_n=\lfloor(n+1)\rho+\theta\rfloor-
     \lfloor n\rho+\theta\rfloor.
\]
The displayed difference equals one precisely when
\(n\rho+\theta<j\le(n+1)\rho+\theta\) for some \(j\in\Z\).
Thus the jump positions are \(\lceil(j-\theta)/\rho\rceil-1\).
For every \(j\) such that \((j-\theta)/\rho\) is not an integer, this
position equals \(\lfloor(j-\theta)/\rho\rfloor\).  Irrationality
allows at most one exceptional \(j\).  Consequently, on a sufficiently
far tail the set of ones agrees with \(S(1/\rho,-\theta/\rho)\), as
required.

\smallskip

Now suppose that \(\rho=p/q\) in lowest terms.  Since
\(a_q\le p\le b_q\) and \(b_q-a_q\le1\), the counts \(C_q(t)\)
all lie in \(\{p-1,p\}\) or all lie in \(\{p,p+1\}\).
Thus the integer deviations \(C_q(t)-p\) are either all nonnegative or
all nonpositive.

Fix a residue \(a\in\{0,\ldots,q-1\}\).  Along the progression
\(a+q\Z_{\ge0}\), these deviations satisfy the telescoping identity
\[
 \sum_{j=0}^{h-1}\bigl(C_q(a+jq)-p\bigr)
 =C_{hq}(a)-hp\qquad(h\ge1).
\]
The interval-count bound shows that the right side has absolute value at
most one.  All summands are integers, and no cancellation is possible.
Hence at most one summand can be nonzero.  There are only \(q\) residue classes, so \(C_q(t)=p\) for
all sufficiently large \(t\).  It follows that
\(w_{t+q}-w_t=C_q(t+1)-C_q(t)=0\) for all such \(t\), proving eventual
\(q\)-periodicity.

Extend this periodic tail to a binary word on \(\Z\).  Every finite
factor of the extension occurs arbitrarily far along the original tail;
hence the extension is balanced.  Its density is \(p/q\), so
Lemma~\ref{lem:periodic-balanced} identifies its support as a rational
Beatty set of modulus \(q/p\).  This proves the required representation
also in the rational case.
\end{proof}

\begin{proof}[Proof of Corollary~\ref{cor:balanced}]

Suppose that a balanced word with densities
\(\rho_1>\cdots>\rho_m>0\) exists.  Since the densities sum to one,
each lies in \((0,1)\).  Apply Lemma~\ref{lem:eventual-beatty} to each
letter indicator.  There are finitely many letters, so their Beatty
representations hold beyond a common threshold.  Thus the resulting sets
\(S(1/\rho_i,\beta_i)\) contain every sufficiently large integer exactly
once.

Their positive tails form an eventual covering family with \(m\ge3\)
pairwise distinct moduli.  Graham's
corollary~\cite[Corollary, p.~358]{Graham1973} therefore implies that
all these moduli are rational.  Choose a common period \(Q\) of the
resulting Beatty sets.  Their covering multiplicity is \(Q\)-periodic;
since it equals one on a tail, it equals one on all of \(\Z\).
Theorem~\ref{thm:main} now gives
\(\rho_i=2^{m-i}/(2^m-1)\) for every \(i\).

\smallskip

For sufficiency, we use the binary construction of Fraenkel~\cite{Fraenkel1973}.
Put \(Q=2^m-1\) and \(P_i=2^{m-i}\), and define
\[
 B_i=S(Q/P_i,1-2^{i-1})\qquad(1\le i\le m).
\]
Since \(Q/P_i=2^i-1/P_i\), the indices \(1\le j\le P_i\) give
\[
 \left\lfloor\frac{jQ}{P_i}+1-2^{i-1}\right\rfloor
 =2^ij-2^{i-1}=2^{i-1}(2j-1).
\]
These are the representatives of \(B_i\) in \(\{1,\ldots,Q\}\).
They are exactly the integers in that interval divisible by \(2^{i-1}\)
but not by \(2^i\), and therefore partition the interval as \(i\)
ranges from \(1\) to \(m\).

The \(Q\)-periodic sets \(B_i\) consequently partition \(\Z\) and
have densities \(P_i/Q\).  Their coding word is balanced by the
interval-count argument in Section~\ref{sec:deletion}.  Restricting it to
the nonnegative integers gives the required one-sided word.
\end{proof}

\section{Concluding remarks}

The structural point of the proof is the passage from Fourier cancellation
to divisor concentration.  A hypothetical subcritical partition makes every
row of the minimum-gcd layer pass, while the minimum-row slack identity forces
nearly all numerator weight onto divisors of \(Q-s\) or \(Q+s\).  For
\(s>1\), the least such divisor creates a fixed gap in a second row; for
\(s=1\), the only persistent replacement is a dyadic predecessor chain,
which is controlled by harmonic propagation.  This isolates the structural
obstruction uniformly in the number of parts.  The rigidity result proved
here is deliberately subcritical: it supplies exactly the conclusion needed
under \(3a<Q\), without asserting the full inverse-sine conjecture.

The three exact finite exclusions in Appendix~\ref{app:verification} close
the bounded cases left by this structural argument.  After the one-third
theorem, the proof is entirely combinatorial: balanced deletion, the
mechanical description of periodic balanced sets, and the rigid two-set
extension give the binary classification by induction.

\section{Acknowledgement}

The authors used GPT-5.6 Sol for resolving the cases \(m=8,9,10,11\), 
which inspired the authors' development of key ideas of the proof strategy, 
and for language polishing in the writing of this paper. 
OpenAI Codex assisted with the design, implementation, and execution of code for finite-case verification. 
\href{https://ziv.ink}{Ziv}, developed by Hu Tan and Xukun Wang, 
was used for supplementary mathematical exploration and checking. 
All mathematical ideas, arguments, proofs, and conclusions were determined and verified by the authors, 
who take full responsibility for the entire content of this work. The finite-case verification code and the accompanying Lean code are publicly available at \url{https://github.com/TanHu1999/fraenkel-conjecture-verification}.

\appendix

\section{Exact finite exclusions}\label{app:verification}

\subsection{Finite verification statements}

Three finite exclusions are logical steps in the proof.  This appendix states
what each program checks and why its output implies the corresponding
mathematical assertion.  Table~\ref{tab:contracts} records each obligation,
the result in which it is used, its input range, and its exact method and
outcome.  The subsequent subsections justify these implications;
Appendix~\ref{app:reproducibility} records the checksums and reproduction
instructions.

\begin{table}[ht]

\centering
\footnotesize
\begin{tabularx}{\textwidth}{@{}
>{\raggedright\arraybackslash}p{0.14\textwidth}
>{\raggedright\arraybackslash}p{0.20\textwidth}
>{\raggedright\arraybackslash}p{0.14\textwidth}
>{\raggedright\arraybackslash}p{0.14\textwidth}X@{}}
\toprule
Verification & Proof obligation & Used in & Input range & Method and result \\
\midrule
\(s>1\) finite box &
Nonexistence of an all-passing subcritical set with \(s>1\) &
Proposition~\ref{prop:s-positive-finite} &
\(7\le Q\le301\), eligible \(s>1\) &
Integer ceilings and exact \(0\)--\(1\) knapsack; \texttt{certificate passed} \\
Critical \(s=1\) refinement &
The critical \(s=1\) branch satisfies \(Q\le289\) &
Proposition~\ref{prop:critical-289} &
odd \(7\le Q\le5607\), eligible nondyadic \(y\) &
Exact \texttt{Fraction} envelopes; all survivors have \(Q\le289\) \\
Global \(s=1\) box &
No all-passing selection system exists for \(4\le Q\le715\) &
Theorem~\ref{thm:s-one} &
\(4\le Q\le715\) &
Integer ceilings, branching, and exact knapsack; zero feasible systems \\
\bottomrule
\end{tabularx}
\caption{Exact finite verifications used in the proof.}
\label{tab:contracts}
\end{table}

The first verifier closes the bounded \(s>1\) branch, the second places the
critical \(s=1\) branch inside the global range, and the third excludes every
remaining \(s=1\) system.  Together with the uniform estimates in the main
text, these three implications prove Theorem~\ref{thm:SC} and hence the
one-third theorem.  The exact constructions, safety arguments, and reported
outputs are recorded below.

\subsection{Correctness of the finite reductions}

The rigor of all three finite exclusions rests on four facts.

\begin{enumerate}[label=(\roman*)]
\item Every true kernel value is strictly below the rational envelope in
Proposition~\ref{prop:kernel}.  In the two integer verifiers, every resulting
rational weight is then rounded \emph{upward}.  Thus a genuine all-passing
row necessarily passes the integer inequality used by those verifiers.

\item Every knapsack is an exact \(0\)--\(1\) maximization under the only
relaxed global condition, the numerator budget.  If even this relaxed
maximum is too small, deleting the vertex is monotone and cannot delete a
vertex of a genuine counterexample.

\item The global \(s=1\) verifier performs both branches whenever monotone
deletion stabilizes.  Hence no feasible binary assignment is omitted.

\item The critical verifier uses Python's normalized rational numbers.  Its
comparisons with one and its reported minimum margin are exact.
\end{enumerate}

These principles justify the implications from the program outputs to
Proposition~\ref{prop:s-positive-finite},
Proposition~\ref{prop:critical-289}, and Theorem~\ref{thm:s-one}.  No other
program output is used as a logical premise in the proof.

\subsection{The \texorpdfstring{\(s>1\)}{s > 1} finite box}
\label{app:s-positive-verifier}

Fix admissible integers \(Q,s\) with \(7\le Q\le301\), and put
\[
 V_{Q,s}=\{a\in\Z:s\le a<Q/3,\ (a,Q)=1\}.
\]
For each \(b\in V_{Q,s}\), let
\(b^{-1}\in\{1,\ldots,Q-1\}\) be the inverse of \(b\) modulo \(Q\).
For distinct \(a,b\in V_{Q,s}\), let \(d_{ab}=\normQ{ab^{-1}}\) and,
with \(L=10^7\), define the upward-rounded integer weight
\begin{equation}\label{eq:rounded-weight}
 w_{ab}=\left\lceil L\left(\frac1{d_{ab}}
              +\frac{16d_{ab}}{7Q^2}\right)\right\rceil.
\end{equation}
For a current pool \(P\subseteq V_{Q,s}\), the relaxed score of \(a\in P\)
is the exact knapsack value
\[
 \operatorname{score}_P(a)=
 \max\left\{
   \sum_{b\in T}w_{ab}:
   T\subseteq P\setminus\{a\},\quad
   \sum_{b\in T}b\le Q-a
 \right\}.
\]
To compute it, list \(P\setminus\{a\}=\{b_1,\ldots,b_t\}\), set
\(C_a=Q-a\), and initialize \(D_0(0)=0\) and
\(D_0(u)=-\infty\) for \(1\le u\le C_a\).  The exact-load recurrence is
\[
 D_j(u)=
 \begin{cases}
  D_{j-1}(u),&u<b_j,\\
  \max\{D_{j-1}(u),D_{j-1}(u-b_j)+w_{ab_j}\},&u\ge b_j.
 \end{cases}
\]
Consequently
\(\operatorname{score}_P(a)=\max_{0\le u\le C_a}D_t(u)\).

The peeling is synchronous.  Starting from \(P_0=V_{Q,s}\), form
\[
 R_k=\{a\in P_k:\operatorname{score}_{P_k}(a)<L\},
 \qquad P_{k+1}=P_k\setminus R_k.
\]
Every score in a round is computed from the fixed set \(P_k\).  If a
genuine all-passing set \(C\) lies in \(P_k\), then
Proposition~\ref{prop:kernel} and \eqref{eq:rounded-weight} show that every
\(a\in C\) has score strictly larger than \(L\).  Thus \(C\subseteq
P_{k+1}\), and deletion of the prescribed minimum \(s\) rules out \(C\).

The script \path{sc_s_gt1_near_certificate.py} checks all \(8{,}916\)
eligible pairs \((Q,s)\) and deletes \(s\) in every case.  It records
\(182{,}653\) deletion events through the decisive rounds.  The smallest
margin for deleting \(s\) is \(7908/10^7\), attained at
\((Q,s)=(97,7)\).

\subsection{The critical rational refinement}
\label{app:critical-verifier}

For every odd \(Q\) with \(7\le Q\le5607\), the script
\path{sc_s1_critical_certificate.py} enumerates all nondyadic integers
\(y\ge5\) satisfying \(3y<Q\) and \(y\mid Q-1\) or \(y\mid Q+1\).
It evaluates the envelope \(U(Q,y)\) from \eqref{eq:U-total} with Python's
normalized \texttt{Fraction} class; no comparison is made in floating-point
arithmetic.

There are \(39{,}232\) eligible pairs.  Exactly \(155\) are not excluded
by the inequality \(U(Q,y)<1\), and every survivor has \(Q\le289\).  In the
decisive range \(Q\ge290\), the smallest positive exclusion margin is
\[
 \frac{1140922873693}{64523567545080},
 \qquad (Q,y)=(311,24).
\]
This is the exact finite assertion used in
Proposition~\ref{prop:critical-289}.

\subsection{The global \texorpdfstring{\(s=1\)}{s = 1} finite box}
\label{app:global-verifier}

The script \path{sc_s1_finite_certificate.py} exhausts the integer system
\eqref{eq:finite-system}.  At a search node, variables are partitioned into
mandatory, excluded, and undecided sets \(I,E,R\), respectively.  To bound
a row indexed by any \(a\notin E\), first force that row to be selected by
putting
\[
 I_a=I\cup\{a\},\qquad
 C_a=Q-\sum_{i\in I_a}i,\qquad
 B_a=\sum_{i\in I_a\setminus\{a\}}v_{ai}.
\]
If \(C_a<0\), this forced choice is impossible.  Otherwise its exact upper
bound over all completions of the node is
\[
 \operatorname{UB}_a=
 B_a+\max\left\{
   \sum_{b\in T}v_{ab}:
   T\subseteq R\setminus\{a\},\quad
   \sum_{b\in T}b\le C_a
 \right\}.
\]
This is the recurrence of Appendix~\ref{app:s-positive-verifier}, now
initialized by \(D_0(0)=B_a\).  In particular, an undecided \(a\) is
temporarily added to the mandatory set both in its budget and in its row
bound.

All weights are computed as the integer ceiling
\[
 v_{ab}=\left\lceil
 \frac{\Lambda(7Q^2+16d_{ab}^2)}{7Q^2d_{ab}}
 \right\rceil.
\]
There is therefore no floating-point evaluation.  Moreover, for
\(4\le Q\le715\) every row score is below \(4.8\times10^9\), well below the
signed \(64\)-bit limit.  A node closes when a mandatory row has upper bound
below \(\Lambda\); an optional vertex with such a bound is safely excluded.
After monotone deductions stabilize, the verifier branches on an undecided
vertex and explores both children.  Induction on the number of undecided
variables therefore shows that every binary assignment is either explored
or discarded by a valid upper bound.

All \(712\) denominators are infeasible.  The proof-tree node histogram is
\[
 \{1:704,\ 3:3,\ 5:4,\ 7:1\},
\]
meaning that \(704\) denominators close at the root, while respectively
three, four, and one denominators require trees with three, five, and seven
nodes.  The largest tree occurs at \(Q=43\).

\section{Reproducibility record}\label{app:reproducibility}

\subsection{Code and checksums}

Appendix~\ref{app:verification} states the three exact exclusions used in the
proof.  This appendix identifies their ancillary sources, gives their
checksums, and records reproduction instructions; it introduces no additional
mathematical hypothesis.

The three verifications in Table~\ref{tab:contracts} are implemented,
respectively, by \path{sc_s_gt1_near_certificate.py},
\path{sc_s1_critical_certificate.py}, and
\path{sc_s1_finite_certificate.py}.  These are the only programs used as
logical proof steps.  They are distributed as ancillary files with the
paper, and their source is short enough for direct inspection.
The first and third programs use NumPy signed \(64\)-bit integers; the middle
program uses Python's exact \texttt{Fraction} class.  Throughout the stated
ranges, the uniform integer bound proved above is below \(5\times10^9\), far
below \(2^{63}-1\).  The two dynamic-programming implementations
explicitly copy the previous layer before each in-place update, making the
\(0\)--\(1\) semantics transparent.

\begin{table}[ht]
\centering
\small
\begin{tabularx}{\textwidth}{@{}>{\raggedright\arraybackslash}p{0.35\textwidth}X@{}}
\toprule
File & SHA--256 \\
\midrule
\path{sc_s_gt1_near_certificate.py} &
\texttt{\seqsplit{76db33224acaae671b1d17f8be1d8dafeb7c493fa3665f7dad66e2a98d29dbfb}}\\
\path{sc_s1_critical_certificate.py} &
\texttt{\seqsplit{97c035f3a8fdc556e9acf0db9ca378590b8a445f2417dfb1a4b3a5fc15ad8e04}}\\
\path{sc_s1_finite_certificate.py} &
\texttt{\seqsplit{af94572b631fbea2356809af2696ceb51c656ca82c0d74f3fcae7ce8e16723f3}}\\
\bottomrule
\end{tabularx}
\caption{Logically necessary exact verifiers.}
\label{tab:certificates}
\end{table}

\subsection{Reproduction instructions}

The reference run used CPython 3.12.13 and NumPy 2.3.5.  From the directory
containing the scripts, the complete commands are
\begin{verbatim}
shasum -a 256 sc_s_gt1_near_certificate.py \
  sc_s1_critical_certificate.py sc_s1_finite_certificate.py
python3 -I sc_s_gt1_near_certificate.py
python3 -I sc_s1_critical_certificate.py
python3 -I sc_s1_finite_certificate.py
\end{verbatim}
The hashes produced by the first command are those in
Table~\ref{tab:certificates}.  Each of the three runs terminates with
\texttt{certificate passed}.  They cover, respectively, all \(8{,}916\)
eligible pairs \((Q,s)\), all \(39{,}232\) eligible critical pairs
\((Q,y)\), and all \(712\) denominators \(4\le Q\le715\).  The smallest
reported margins are \(7908/10^7\) in the \(s>1\) box and
\[
 \frac{1140922873693}{64523567545080}
\]
in the decisive range \(Q\ge290\) of the critical refinement.  The three
commands above reproduce these counts and margins directly.

\end{document}